\documentclass{article}
\usepackage[utf8]{inputenc}     
\usepackage[english]{babel}     
\usepackage[a4paper, left=1.5in, right=1.5in, top=1.5in, bottom=1.5in]{geometry}
\usepackage{graphicx}           
\usepackage{amsmath,amssymb}    
\usepackage{amsthm}             
\usepackage{mathtools}          
\usepackage{enumitem}           
\usepackage{listings}           
\usepackage{todonotes}          
\usepackage{hyperref}           
\usepackage{cleveref}
\usepackage{physics}
\usepackage{dsfont}
\usepackage{bbold}
\usepackage{pgfplots}
\usepackage{ulem}
\pgfplotsset{compat=1.15}
\usepackage{mathrsfs}
\usepackage{algorithm2e}
\usepackage{caption}
\usepackage{subcaption}
\usepackage{float}
\usepackage{mwe}
\usepackage{cancel}

\usepackage[noblocks]{authblk}

\usepackage[nottoc]{tocbibind}
\RestyleAlgo{ruled}
\DeclarePairedDelimiterX\set[1]\lbrace\rbrace{#1}

\newtheorem*{theorem*}{Theorem}
\newtheorem{theorem}{Theorem}[section]

\newtheorem{lemma}[theorem]{Lemma}
\newtheorem{proposition}[theorem]{Proposition}

\theoremstyle{definition}
\newtheorem{definition}[theorem]{Definition}

\newtheorem{problem}[theorem]{Problem}
\newtheorem*{problem*}{Problem}

\newtheorem{remark}{Remark}[theorem]

\def \a  {\alpha}

\def \b  {\beta}
\def \g  {\gamma}

\def \d  {\delta}
\def \D  {\Delta}
\def \e  {\varepsilon}

\def \f  {\varphi}

\def \om {\omega}
\def \Om {\Omega}

\def \r  {\rho}
\def \t  {\tau}

\def \del {\nabla}
\def \dd {\mathrm{d}}
\def \p  {\partial}
\def \R  {\mathds{R}}
\def \N  {\mathbb{N}}

\def \Maxi {\mathrm{M}}
\def \Mini {\mathrm{m}}
\def \Linf {L^{\infty}(\Om)}
\def \H {H^1_0\left(\Om \right)}
\def \Phib {\Breve{\Phi}}
\def \dx {\text{d}x}

\newcommand{\brac}[1]{\left(#1 \right)}
\newcommand{\posi}[1]{\left[#1 \right]_+}
\newcommand{\nega}[1]{\left[#1 \right]_-}
\newcommand{\Lom}[1]{L^{#1}(\Omega)}

\newcommand{\infnorm}[1]{\norm{#1}_{L^{\infty}(\Omega)}}

\numberwithin{equation}{section}

\newcounter{AssBio}
\stepcounter{AssBio}

\newcounter{AssEx}
\stepcounter{AssEx}

\title{Robust time-discretisation and linearisation schemes for singular and degenerate evolution systems modelling biofilm growth}
\author{R.K.H. Smeets\thanks{email: \href{mailto:r.k.h.smeets@uva.nl}{r.k.h.smeets@uva.nl}}}
\affil{\small University of Amsterdam, Korteweg-de Vries Institute for Mathematics, The Netherlands}
\author{K. Mitra}
\affil{Eindhoven University of Technology, Department of Mathematics and Computer Science, The Netherlands}
\author{I.S. Pop}
\affil{Hasselt University, Faculty of Science, Belgium}
\author{S. Sonner}
\affil{Radboud University, IMAPP - Mathematics, The Netherlands}

\date{\today}

\begin{document}

\maketitle

\begin{abstract}
   We propose and analyse numerical schemes for a system of quasilinear, degenerate evolution equations modelling biofilm growth as well as other processes such as flow through porous media and the spreading of wildfires.  
   The first equation in the system is parabolic and exhibits degenerate and singular diffusion, while the second 
   is either uniformly parabolic or an ordinary differential equation. 
   First, we introduce a semi-implicit time discretisation that has the benefit of decoupling the equations.  
   We prove the positivity, boundedness, and convergence of the time-discrete solutions to the time-continuous solution. Then, we introduce an iterative linearisation scheme to solve the resulting nonlinear time-discrete problems. Under weak assumptions on the time-step size, we prove that the scheme converges irrespective of the space discretisation and mesh. Moreover,
   if the problem is non-degenerate, the convergence becomes faster as the time-step size decreases. Finally, employing the finite element method for the spatial discretisation, we study the behaviour of the scheme, and  
   compare its performance to other commonly used schemes. These tests confirm that the proposed scheme is robust and fast.
   \\

   
\end{abstract}

\section{Introduction} \label{sec: 1}

\subsection{Motivation}\label{sec: 1 - biofilm}
Let $T > 0$ be a maximal time and $\Om \subset \R^d$ ($d \in \N$) be a bounded Lipschitz domain. With $Q=\Omega\times (0,T]$ denoting the parabolic space-time cylinder, we consider 
the following class of degenerate quasilinear parabolic systems
\begin{subequations} \label{eq: main}
\begin{align}
&\p_t u=\D \Phi(u) + f(v)u, \label{eq: main 1}\\
&\p_t v= \mu \nabla \cdot \left(D(u)\nabla v \right) + g(u,v) \label{eq: main 2}
\end{align}
\end{subequations}
in $Q$.
The first equation \eqref{eq: main 1} describes evolution of a population density $u$ and exhibits degenerate and possibly also singular diffusion leading to the formation of free boundaries propagating at a finite speed. More specifically, the monotone function $\Phi$ vanishes as $u$ tends to zero and can in addition have a singularity as $u$ tends to its maximum value. The diffusion coefficient  
$D$ in the second equation which describes evolution of a substrate concentration $v$ is assumed non-degenerate, i.e. it is bounded above and below by positive constants. 
However, the mobility coefficient  $\mu$ 
appearing in \eqref{eq: main 2} may be either 0 or 1, leading 
to either a coupled system of a parabolic equation and an ordinary differential equation (if $\mu=0$), or two  parabolic equations (if $\mu=1$). The growth and spreading of the population $u$ might depend on multiple substrates. Nevertheless, for simplicity, we consider only one substrate $v$  in \eqref{eq: main 2}, as
an extension to multiple substrates is straightforward.
The system is completed by the initial conditions $u(\cdot, 0) = u_0$ and $v(\cdot, 0) = v_0$ for given functions $u_0, v_0$, and by homogeneous Dirichlet boundary conditions for $u$ and, if $\mu \ne 0$, also for $v$. 
Analysing system \eqref{eq: main} analytically and numerically is challenging due to the 
degenerate and singular diffusion in the first equation which leads to free boundaries and steep gradients, and the nonlinear coupling with the second equation.

The main motivation for this work comes from the biofilm growth models in \cite{EBERLPDEODE, EBERLPDEPDE}, where  
the solution component $u$ in \eqref{eq: main 1} describes the (normalized) biomass density whose evolution is dictated by the diffusion operator
\begin{align}
\Delta\Phi(u)=\nabla\cdot\left(\frac{u^\alpha}{(1-u)^\beta}\nabla u\right) \label{eq: biof 1},\qquad \text{ for some } \qquad \alpha,\beta\geq 1.
\end{align}
With this, \eqref{eq: main 1}   
is coupled to a reaction-diffusion partial differential equation (PDE) or an ordinary differential equation (ODE) modelling the evolution of  
the growth-limiting nutrient concentration $v$.
In the resulting model, the biofilm occupies the region where 
$\{u>0\}$. 
Observe that the  
biomass diffusion in
\eqref{eq: biof 1} shows a degeneracy of porous-medium type as $u$ approaches $0$, ensuring a finite speed of 
propagation of the interface between the biofilm and the surrounding region, as well as a singularity as $u$ approaches $1$. The latter implies that the solution $u$ remains bounded by a constant strictly less than 1 despite the growth term $f$ in the equation. The second equation, \eqref{eq: main 2}, describes the evolution of the nutrient concentration.
The case $\mu=0$ corresponds to immobile substrates (e.g., in the case of cellulolytic biofilms, \cite{EBERLPDEODE}), while $\mu=1$ corresponds to diffusive substrates (e.g., whenever biofilms grow in an aqueous medium, \cite{EBERLPDEPDE}).
The biofilm growth model  
was also extended to take multiple substrates into account, both mobile and immobile, like in \cite{emerenini2016mathematical,mitra2021existence}. The results of our paper generalise directly to these cases. 
More complex multi-species biofilm models including cross-diffusion have been studied in \cite{DaMiZa,RaSuEb}. 

Systems of the form \eqref{eq: main} are not limited to models for biofilm growth. For instance, coupled systems of parabolic and ordinary differential equations appear in the modelling of two-phase or unsaturated flow through porous media when effects like hysteresis or dynamic capillarity are taken into account 
\cite{mitra2021existence,mitra2019wetting} (degenerate, but non-singular diffusion). They also appear in wildfire models \cite{reisch2023analytical} (nonlinear but non-degenerate diffusion), and reaction, diffusion, and adsorption/desorption models in a porous medium \cite{Helmig}, to name a few other applications. 

The aim of this paper is to develop robust, efficient, and structure-preserving time discretisation and linearisation methods for System \eqref{eq: main}  relying on minimal regularity assumptions, and converging even for degenerate and singular diffusion. In what follows, the time discretisation is introduced, as well as the linear iterative schemes. Furthermore, the main results concerning the stability and convergence of these schemes are stated, their proofs being given in the subsequent sections.  

\subsection{Time discretisation} \label{sec: 1 - time}

To define the time discretisation, we take $N \in \N$ and let $\t = \frac T N$ be the time-step size, which is chosen uniform for the ease of presentation. With $n \in \{0, \dots, N\}$, let $t_n = n \t$ be the time-step and denote by $u_n$ the approximation of $u$ at $t = t_n$, and similarly for $v_n$. Then, we use an Euler semi-implicit approach for the time discretisation. All terms in \eqref{eq: main} are discretised implicitly except for the 
reaction functions $f$ and $g$, which are discretised semi-implicitly.  
\begin{problem}[Semi-implicit time discretisation]\label{prob: time-discrete}
Given the approximate solutions $u_{n-1}$ and $v_{n-1}$ at time $t_{n-1}$, find the approximate solution pair $(u_n, v_n)$ at the next time step $t_n$ by solving the following system
\begin{subequations} \label{eq: time-discrete}
\begin{align}
&\frac{1}{\t}(u_n-u_{n-1})=\D  \Phi(u_n) + f(v_{n-1}) u_n, \label{eq: time-discrete 1}\\
&\frac{1}{\t}(v_n-v_{n-1})=\mu \nabla \cdot \left(D(u_n)\nabla v_n \right) + g(u_n, v_{n-1}). \label{eq: time-discrete 2}
\end{align}
\end{subequations}
\end{problem}
This approach has several advantages over explicit and implicit discretisations. Explicit methods lead to a loss of regularity of the time-discrete solutions, which results in instability due to the already low regularity triggered by degenerate diffusion. On the other hand, implicit schemes have the advantage that only a \textit{weak restriction on the time-step size}, independent of the space-discretisation,  is needed to guarantee stability.
Fully implicit time-discretisation schemes for the mentioned biofilm models were analysed in \cite{BALSACANTO2017164, duvnjak2006time, ghasemi_sonner_eberl_2018, HELMER2023386}. However, for fully implicit schemes, the two time-discrete equations originating from \eqref{eq: main} are coupled which makes them challenging to solve especially using an iterative method. In the semi-implicit approach \eqref{eq: time-discrete} the two equations are decoupled. This allows us to solve them sequentially instead of iteratively, i.e. we first solve for $u_n$ and then update $v_n$ using the known $u_n$. In fact, given $u_n$, \eqref{eq: time-discrete 2} is a linear problem for $v_n$ 
which can be solved rather easily. Moreover, with the proposed semi-implicit discretisation we retain the same accuracy and stability one would expect from fully implicit discretisations. Generalising the results in \cite{duvnjak2006time} for the  scalar equation \eqref{eq: main 1}, we show that
under weak assumptions on~$\t$, the discretisation \eqref{eq: time-discrete} 
is well-posed, the time-discrete solutions preserve positivity, remain bounded, and converge to the time-continuous solutions as $\t\to 0$. 
The exact results are stated in Theorems \ref{thm: main time 1} and \ref{theo:time_convergence}.
Below, we summarise them omitting technical details.

\begin{theorem*}[Well-posedness, boundedness, and convergence of the time-discrete solutions] 
For all sufficiently small time steps $\t$, there exists a unique weak solution of \eqref{eq: time-discrete}. 
Moreover, the time-discrete solutions $u_n,v_n$ are positive and $u_n$ is bounded 
almost everywhere in $\Om$. In particular, if $\Phi$ is singular, $u_n$ is bounded by a constant strictly less than the singularity. 
Lastly, the time-discrete solutions converge to the time-continuous solutions as $\t\to 0$.
\end{theorem*}
Concerning space discretisations, we also restrict ourselves to mentioning works addressing specifically System \eqref{eq: main}. In this respect, a finite difference method was used in \cite{Eberl2007finitediff}. The finite volume method was considered in \cite{HELMER2023386}, and the convergence of the space-time dicretisation scheme was proven using an entropy formulation.
Convergence results of mixed finite elements for a variation of the PDE-PDE model of biofilms were shown in \cite{alhammali2024numerical}, and for a PDE-ODE model in the context of porous media flow in \cite{cao2019error}. For the numerical results presented here, we use 
finite elements. However, the numerical analysis is done in a time-discrete, but continuous-in-space setting. Therefore, the results are independent of the chosen spatial discretisation.

\subsection{Linearisation} \label{sec: 1 - linear}

The time-discrete system \eqref{eq: time-discrete} 
is nonlinear, degenerate and singular. For approximating its numerical solution, stable iterative linearisation schemes are needed. Most of the works addressing such type of problems are focusing on a \textit{direct approach}. More precisely,  for solving \eqref{eq: time-discrete 1}, $w=\Phi(u)$ is considered as the primary unknown yielding $u=\Phi^{-1}(w)$ which is then linearised by expanding in terms of the last iterate. With reference to \eqref{eq: time-discrete 1}, this approach is convenient because there are no spatial derivatives applied to the nonlinear terms. However, this requires that $\Phi$ is an invertible function; if this is not the case, then a regularisation step is required. Alternatively, one can use $u$ as the primary variable, and avoid inverting the function $\Phi$. In this case, the nonlinearity appears under the Laplace operator, which makes the construction of robust linearisation schemes and, in particular, proving the convergence a complex task.     

Following \cite{cances2021error} here we consider instead a \textit{split formulation} involving two primary unknowns,  $u$ and $w$, which are related through the algebraic relationship $w = \Phi(u)$. We reformulate the time-discrete version of \eqref{eq: main 1} as a system of a linear elliptic equation and an algebraic one. For this reformulated system, we construct linear iterative schemes having a stable and robust behaviour. They all fit in the general framework given below. 

\begin{problem}[Linearisation scheme] \label{prob: linear}
For $i \in \N$, let $u^i_n$ be the $i^{\rm th}$ iterate of the $n^{\rm th}$ time-step, and let $u^0_n:=u_{n-1}$ be given. To compute $u^i_n$ from $u^{i-1}_n$,  first solve for $(\tilde{u}^i_n,w^i_n)$ satisfying
\begin{subequations}\label{eq: linear}
\begin{align}
&\frac{1}{\t}(\Tilde{u}^i_n-u_{n-1})=\D  w^i_n + f(v_{n-1}) \Tilde{u}^i_n, \label{eq: linear 1}\\
&L^i_n(\Tilde{u}^i_n-u^{i-1}_n)=w^i_n-\Phi(u^{i-1}_n).  \label{eq: linear 2} \end{align}
The factors $L^i_n$ will be specified below. The way they are chosen is defining the different schemes used here. Finally, we take the positive part of $\Tilde{u}^i_n$,  
\begin{align}
&u^i_n = \posi{\Tilde{u}^i_n}.
\end{align}
\end{subequations}
\end{problem}

\noindent
If $w_n^i$ and $\tilde{u}_n^i$ (and, consequently, also $u^i_n$)  converge to $w_n$ and $u_n$ respectively, then the limits satisfy  $w_n=\Phi(u_n)$, and $u_n$ is a (weak) solution of \eqref{eq: time-discrete 1}. Observe that the formulation used in \Cref{prob: linear}, where \eqref{eq: time-discrete 1} is split into a linear elliptic equation and a nonlinear algebraic one, is well suited for degenerate problems. In particular, no regularisation is needed in the slow diffusion case, i.e. when $u^i_n$ takes values for which $\Phi^\prime$ vanishes. 

As mentioned, the choice of the factors $L^i_n$ in \eqref{eq: linear 2} leads to different linearisation schemes. With $L^i_n = \Phi^\prime(u^{i-1}_n)$, one obtains the Newton Scheme (NS) in the context of the splitting formulation \eqref{eq: linear}. The convergence is guaranteed rigorously in the fully discrete case, but this depends strongly on the spatial discretisation and mesh size \cite{brenner2017improving,Park,NS3}. For time-dependent problems, since the initial guess is often the solution at the previous time step, this implies that the time-step size should be sufficiently small, which may cancel the advantages brought by the implicit discretisation. In the same category, we mention the modified Picard scheme \cite{PS1} and the J{\"a}ger--Ka{\v c}ur scheme \cite{Jager1991_92, Jager1995}, for which the linear convergence can be proven rigorously under similar conditions as for the NS. 

Ideally, one works with a scheme that is unconditionally convergent, i.e. the time-step~$\t$ does not depend on the spatial discretisation and the mesh-size~$h$, which was one of the main reasons for choosing an implicit time discretisation. This can be achieved by using the L-scheme (LS) \cite{PS2, POP2004365}, which is nothing but choosing $L^i_n = L$ (a suficiently large constant) in \eqref{eq: linear 2}. This scheme has a guaranteed but linear convergence,  irrespective of the initial guess and the spatial discretisation. These results were  extended in \cite{RaduKumar} to doubly-degenerate problems, where a H{\"o}lder continuous, not necessarily strictly increasing nonlinearity appears under the time derivative. The drawback of this scheme is its significantly slower convergence when compared to the NS (whenever the latter converges) \cite{PS2, SEUS2018331}. Improvements can be made by choosing $L$ adaptively in each step, or by performing first a number of LS steps, and then switching to the NS entirely when the iterations are close enough to the solution \cite{PS2, Stokke2023}. 

Such observations have lead to the Modified L-scheme, or M-scheme (MS) for short, as introduced in \cite{MITRA20191722}, and on which we mainly focus here. We define $L^i_n \coloneqq \max\{\Phi'(u^{i-1}_n) \\+ M\t^\g,\, 2M \t^\g\}$, for some $\g\in (0,1]$ and sufficiently large $M>0$. Hence, the MS can be viewed as a combination of the NS and the LS where the LS is a first order global method ensuring stability, while the NS is a first order local method speeding up convergence. In particular, the MS has a first order local term to speed up convergence but regularizes it with a second order global term which captures the evolution of $u_n$ to ensure stability. In an earlier work \cite{MITRA20191722}, it has been shown for the direct formulation that the MS is indeed unconditionally stable while achieving much better convergence rates than the LS. 

In this work, we apply the LS and MS to the splitting formulation in \Cref{prob: linear}, allowing us to handle systems with porous medium type degeneracies, as well as singular diffusion. Similarly to \cite{MITRA20191722}, we prove that the MS converges unconditionally and that, in the non-degenerate case, it has a contraction rate that scales with $\t$. The two main results for the LS and MS in \Cref{sec: linear} are summarised in the following theorems.

\begin{theorem*}[Convergence of the L-scheme] 
For a sufficiently small time step size $\t$ independent of the mesh size $h$, the L-scheme converges to a function $u_n$ that is the weak solution of our time-discretised \cref{eq: time-discrete 1}. In the non-degenerate case, the L-scheme converges with a contraction rate $\alpha < 1$. 
\end{theorem*}

\begin{theorem*}[Convergence of the M-scheme]
For a sufficiently small time step size $\t$ independent of the mesh size $h$ and under certain boundedness conditions on the iterates, the M-scheme converges to a function $u_n$ that is the weak solution of our time-discretised solution \cref{eq: time-discrete 1}. In the non-degenerate case, the M-scheme converges with a contraction rate $\alpha <1 $, that scales with $\t^\g$ for some $\gamma\in(0,1]$.
\end{theorem*}

 The outline of our paper is as follows: In  \Cref{sec: prelim} we provide the required functional setting and background and state the structural assumptions. In \Cref{sec: time-discrete} we prove the results for the time discretisation and in \Cref{sec: linear} we show the convergence results for both the LS and the MS. In \Cref{sec: alg} we perform numerical simulations and compare the performances of the NS and MS for a porous medium equation and both cases of the biofilm model (PDE-PDE system and PDE-ODE system). 
Finally, in \Cref{sec: concl and disc} we summarize our results and discuss potential future research.

\section{Preliminaries} \label{sec: prelim}

\subsection{Functional setting and background}
Let $\Om \subset \R^d$ be a bounded Lipschitz domain. The corresponding $L^2$ inner product and norm are denoted by $\brac{\cdot, \cdot}$ and $\norm{\cdot}$, and norms with respect to other Banach spaces $V$ by $\norm{\cdot}_V$. 
By $\left(W^{1,p}(\Om),\|\cdot\|_{W^{1,p}(\Om)}\right)$, $1\leq p < \infty$, we denote the Sobolev spaces
and use the short-hand notation $H^1(\Om) \coloneqq W^{1,2}(\Om)$. 
The space $H^1_0(\Om)$ is the closure of $C^{\infty}_c(\Om)$ in $H^1(\Om)$, which is equipped with the equivalent norm $\norm{u}_{H^1_0(\Om)} \coloneqq \norm{\nabla u}$ due to the Poincar\'e inequality 
\begin{equation}\label{eq:Poincare}
    \norm{u} \leq C_{\Om} \norm{\nabla u} \quad \text{for all } u \in H^1_0(\Om), 
\end{equation}
where $C_\Omega>0.$
The dual space of $\H$ is denoted by $H^{-1} \coloneqq \brac{\H}^\ast$ with the norm
\begin{equation}
    \norm{u}_{H^{-1}\Om} \coloneqq \sup_{\phi \in \H} \frac{\langle u,\phi\rangle}{\norm{\nabla \phi}},
\end{equation}
where $\langle\cdot,\cdot\rangle$ denotes the duality paring.
Since we consider homogeneous Dirichlet boundary conditions, we will mainly use $H^1_0(\Om)$. Lastly, we consider the Bochner spaces $L^p(0,T;V)$, with $V$ a Banach space, equipped with the norm 
\begin{equation}
    \norm{u}_{L^p(0,T;V)} \coloneqq \brac{\int_0^T \norm{u(t)}_{V}^p dt}^{1/p} < \infty.
\end{equation}

We will frequently use Young's inequality
\begin{equation} \label{eq: Young}
    uv \leq \frac{1}{2\r}u^2 + \frac{\r}{2}v^2, \quad \text{for }\r>0 \text{ and } u,v \in \R,
\end{equation}
the Cauchy-Schwarz inequality 
\begin{equation}
    \left|\int_\Om uv \right| \leq \norm{u} \norm{v},\qquad  u,v\in L^2(\Om), \label{eq: cauchy schwarz}
\end{equation}
and the discrete Gronwall Lemma: Let $\{u_n\}_{n\in \N}, \{a_n\}_{n\in \N},\{b_n\}_{n\in \N}$ be non-negative sequences such that $u_n \leq a_n + \sum_{k=1}^{n-1}b_k u_k$, then
\begin{equation} \label{eq: discr Gronwall}
    u_n \leq a_n + \sum_{k=1}^{n-1}a_k b_k \exp\left(\sum_{k<j<n} b_j\right).
\end{equation}

Lastly, for convex functions $\Psi \in C\brac{\R^+}$ with $\Psi(0) = 0$, we have the following inequalities
\begin{subequations}
    \begin{align}
        &\text{Jensen's inequality:} \quad \Psi\brac{\frac{1}{|\Om|}\int_{\Om}|f|} \leq \frac{1}{|\Om|}\int_{\Om}\Psi(|f|) \quad \text{for } f\in L^2(\Om), \label{eq: jensen}\\
        &\text{Super-additivity:} \quad  \Psi(a) + \Psi(b) \leq \Psi(a+b) \quad  \text{for all } a,b \geq 0. \label{eq: superadd}
    \end{align}
\end{subequations}

To denote the positive and negative part we write $\posi{u} \coloneqq \max\{0, u\}$ and $\nega{u} \coloneqq \min\{0, u\}$ and use the notation $a \lesssim b$ if $a \leq C b$ for some constant $C>0$. 
Moreover,  $\R^+=\{x \in \R \; \colon \; x \geq 0\}$ and $\R^+_\ast=\{x \in \R \; \colon \; x > 0\}$. In some proofs we will also use the notation 
\begin{equation}
    I(a,b) = \{x \; \colon \; \min\{a,b\} \leq x \leq \max\{a,b\} \}, \qquad a,b\in\R.
\end{equation}
If  $u,v$ are two given functions then $I(u,v)$ should be considered pointwise a.e. 

Finally, throughout this work, $C>0$ will denote a generic constant that might change in each occurrence and from line to line but will always be independent of $\t$.

\subsection{Structural assumptions} \label{sec: assump}
Dependng on the function $\Phi$ in \eqref{eq: biof 1}, the problems may be singular-degenerate, or have a structure resembling the porous medium equations. To cover both cases, we introduce a maximum density/concentration value $b$, where $b = 1$ in the former case (the biofilm model) and $b = \infty$ for porous medium type equations. 

With $b$ as above, We make the following structural assumptions for \Cref{eq: main}
\begin{enumerate}[label=(P\theAssBio)]
\setlength\itemsep{-0.2em}
\item $\Phi \colon [0,b)\to \R^+$ is an increasing function with locally Lipschitz continuous derivatives, satisfying  
\begin{align*}
& \quad \Phi(0)=0, \quad \lim_{u\nearrow b}\Phi(u)=\infty, \text{ and } \Phi'(u)>0  \text{ for } u\in (0,b),
\end{align*}
with $\inf_{[0,b]}\Phi'=:\phi_\Mini\in [0,\infty)$, and $\sup_{[0,b]}\Phi'=:\phi_\Maxi\in (0,\infty]$. We furthermore require that $\Phi^{\prime}$ is strictly increasing in $[0,\varepsilon_0)$ for some $\varepsilon_0 \in (0,b)$.

\label{ass: 1} \stepcounter{AssBio}

\item $f\colon\R^+ \to \R$ is Lipschitz continuous and bounded, with $\|f\|_\infty = f_\Maxi$ for some constant $f_\Maxi\geq 0$. $g\colon[0,b)\times\R^+ \to\R$ is   Lipschitz continuous with Lipschitz constant $g_\Maxi>0$. Moreover, we assume that $g(\cdot,0)\geq 0$. 
\label{ass: 2} \stepcounter{AssBio}
\end{enumerate}

\begin{enumerate}[label=(P\theAssBio)]
\setlength\itemsep{-0.2em}

\item $D\colon[0,b)\to \R$ is a  continuous function. There exist constants $D_\Mini, D_\Maxi$ s.t. $0<D_\Mini \leq D(u) \leq D_\Maxi<\infty$ for all $u \in [0, b)$. \label{ass: 3} \stepcounter{AssBio}
\end{enumerate}

The functions $\Phi$, $f$, $g$, $D$ are extended for arguments $u\in \R^-$ by their values at $u=0$. If $b=1$, the functions that are bounded at $u=b$ are also extended for $u \in [b,\infty)$ by their values at $u=b$.

\begin{remark}[Validity of the assumptions \ref{ass: 1} - \ref{ass: 3}]
The biofilm models \cite{EBERLPDEODE, EBERLPDEPDE} (see Equation \eqref{eq: PDE-PDE ex num} in Section \ref{sec: alg}) and the porous medium equation satisfy the assumptions \ref{ass: 1} - \ref{ass: 3}, with $b=1$ and $b = \infty$ respectively. Note that we only consider non-negative solutions as $u$ and $v$ denote densities and/or concentrations. 
\end{remark}

For the initial data we assume the following.
\begin{enumerate}[label=(P\theAssBio)]
\setlength\itemsep{-0.2em}
\item The initial conditions $u_{0}, v_{0} \colon\Om \to [0, \infty)$ are s.t. $v_{0} \in L^2(\Om)$, $u_{0} \in L^\infty(\Om)$ and $\|u_{0}\|_{\infty} < b$.  
\label{ass: 4}\stepcounter{AssBio}
\end{enumerate}
Finally, for the ease of presentation we consider homogeneous Dirichlet boundary conditions for $u$, and also for $v$ if $\mu=1$. The results here can be the extended to mixed Dirichlet-Neumann boundary conditions, following the ideas in \cite{HISSINKMULLER,mitra2023wellposedness}.

\subsection{Weak formulation of the time continuous problem} \label{sec: time cont}

\noindent 
We consider weak solutions of \Cref{eq: main} with homogeneous Dirichlet boundary conditions and initial data satisfying \ref{ass: 4}.
\begin{definition}[Weak formulation]
A weak solution of \eqref{eq: main} is a pair $(u,v)\in \\ C([0,T];L^2(\Om))^2 \allowbreak
\cap \, H^1(0,T;H^{-1}(\Om))^2$ s.t. $(\Phi(u),\mu v)\in L^2(0,T;\H)^2$, $(u,v)(0)=(u_0,v_0)$, and 
\begin{subequations}
\begin{align}
    \int_0^T\left\langle \p_t u,\phi\right\rangle +\int_0^T (\nabla \Phi(u), \nabla \phi) &= \int_0^T (f(v)u, \phi) 
    , \nonumber \\
    \int_0^T\left\langle \p_t v, \eta\right\rangle  + \int_0^T \mu(D(u) \nabla v, \nabla \eta) &= \int_0^T (g(u, v), \eta) 
     \nonumber
\end{align}
\end{subequations}\label{def:weak_sol}
hold for all $\phi, \eta \in L^2(0,T;\H)$.  
\end{definition}

For the well-posedness of Problem \eqref{eq: main} with $b=\infty$ and $\mu=1$,  we refer to \cite{alt1983quasilinear,otto1996l1}. If $\mu=0$, the existence and uniqueness of solutions for the coupled system follow by $L^1$-contraction similarly as in \cite{mitra2023wellposedness}. Well-posedness results for the system with $b=1$ and either Dirichlet or mixed Dirichlet-Neumann boundary conditions were obtained in \cite{ HISSINKMULLER, mitra2023wellposedness}. In particular, uniqueness can be shown if $D$ in \ref{ass: 3} is independent of $u$ \cite{HISSINKMULLER} or if $\mu=0$ \cite{mitra2023wellposedness}. Under these assumptions, local well-posedness was also shown for homogeneous Neumann boundary conditions in \cite{mitra2023wellposedness}. 
Furthermore, it was shown that solutions are non-negative, and if \ref{ass: 4} holds, then the solution $u$ is bounded by a constant strictly less than $1$, i.e. the singularity in the diffusion coefficient is not attained.
The local H\"older continuity of solutions of such systems was studied in \cite{HISSINKMULLER2}. 
In this direction,  we also mention \cite{efendiev2009existence} 
where the specific  PDE-PDE  biofilm model \cite{EBERLPDEPDE} (corresponding to $\mu = 1$ with diffusion coefficient \eqref{eq: biof 1}) was analyzed, 
\cite{mitra2021existence} where the existence of solutions for a similar degenerate PDE-ODE system was studied, and \cite{barrett2007existence} where a doubly degenerate PDE-PDE system was analysed.

Lastly, we remark that, for simplicity, we assume homogeneous Dirichlet boundary conditions. Extending the results to mixed Dirichlet-Neumann boundary conditions is possible following the arguments in \cite{HISSINKMULLER, mitra2023wellposedness}.

\section{Time discretisation} \label{sec: time-discrete}
\noindent In this section we analyse the following weak form of the time discretised system \eqref{eq: time-discrete}. We first state it in a weak form.  

\begin{problem*}[Weak formulation of the time-discretised system]
Let $n \in \{1, \dots, N\}$ and $u_{n-1},\,v_{n-1}\in L^2(\Om)$ given. Find 
$(u_n,v_n) \subset L^2(\Om)^2$ such that  $\Phi(u_n),\, \mu v_n \in \H$, and for all $\phi, \eta \in \H$ it holds 
\begin{subequations} \label{eq: time-discrete weak system}
    \begin{align}
        \left(\frac{1}{\t}(u_n - u_{n-1}),\phi\right) + (\nabla \Phi(u_n), \nabla \phi) &= (f(v_{n-1})u_n, \phi)  
        , \label{eq: time-discrete weak 1} \\
        \left(\frac{1}{\t}(v_n - v_{n-1}), \eta\right) +  \mu(D(u_n) \nabla v_{n}, \nabla \eta) &= (g(u_n, v_{n-1}), \eta) 
        . \label{eq: time-discrete weak 2}
    \end{align}
\end{subequations}
\end{problem*}

Throughout this paper, we write $w_n=\Phi(u_n)$ and use the shorthand notation
\begin{align}\label{eq:def_hn}
    h_{n-1}\coloneqq 1-\t f(v_{n-1}).
\end{align}
Note that $h_{n-1}$ is positive if $\t<1/f_\Maxi$. 

\begin{remark}[The decoupling of the equations]\label{rem: decouple equations}
    Observe that the solution $u_n\in L^2(\Om)$ of \eqref{eq: time-discrete weak 1} does not depend on the solution $v_n\in L^2(\Om)$ of \eqref{eq: time-discrete weak 2}. Hence, the system \eqref{eq: time-discrete weak system} can be solved sequentially, i.e. we first solve the nonlinear problem \eqref{eq: time-discrete weak 1} and subsequently the linear problem \eqref{eq: time-discrete weak 2}.
\end{remark}

\noindent In this section we prove the following results, already briefly mentioned in \Cref{sec: 1 - time}.

\begin{theorem}[Well-posedness and boundedness of the time-discrete solutions]\label{thm: main time 1}
For $\t< 1\slash f_\Maxi$, there exists a unique weak solution $(u_n,v_n)$ of \eqref{eq: time-discrete weak system}. 
Moreover, there exist  $\t_{\rm disc} \coloneqq \min\{1/f_\Maxi, 1/g_\Maxi\} > 0$ and $\Breve{u}\in [0,b)$ independent of $n$, such that
\begin{align}\label{eq:u_Linf_bdd}
    0\leq u_n \leq \Breve{u}, \quad \text{ and } \quad  0\leq v_n \quad  \text{ a.e. in } \Om \quad  \text{ for all } 1\leq n\leq N \text{ and } \t<\t_{\rm disc}. 
\end{align}
In fact, $\Breve{u}$ is given by
\begin{align}\label{eq:uup}
    \Breve{u}= \begin{cases}\|u_0\|_{L^\infty}  \exp\brac{\frac{Tf_\Maxi}{1-\t f_\Maxi}} &\text{ if } b=\infty,\\[.5em]
\Phi^{-1}\left( \|\Phi(u_0)\|_{L^\infty} + \frac{{\rm diam}(\Om)^2}{2d} f_\Maxi\right)<1 &\text{ if } b=1,
\end{cases}
\end{align}
\end{theorem}

\begin{remark}[Computable upper bound for $u_n$]\label{rem: computable}
    Observe that \eqref{eq:uup}--\eqref{eq:u_Linf_bdd} provides a uniform upper bound for $u_n$  that is a priori computable. This will be used in \Cref{sec: linear} to show the convergence of the iterative linearisation scheme.
\end{remark}

\begin{theorem}[Convergence of the time-discrete solutions] \label{theo:time_convergence}
Let $(u,v)\in C([0,T];L^2(\Om))^2$ be the unique weak solution of \Cref{eq: main}. For a time-step size $\t=\frac{T}{N_\tau}>0$, $N_\t\in \N$, let $\{(u_n,v_n)\}_{n\in \N}\subset (L^2(\Om))^2 $ be the time-discrete solution of \eqref{eq: time-discrete weak system} with $\{w_n\}_{n\in \N}\subset \H$. Then, in addition to \ref{ass: 4}, if $u_0\in H^1_0(\Om)$, then along any sequence of $\t$ converging to 0 we have
\begin{align}\label{eq:Conv-H1}
    \sum_{n=1}^{N_\t} \int_{(n-1)\t}^{n\t} \left[\|u_n - u(t)\|^2 + \|w_n-\Phi(u(t))\|^2 + \|v_n-v(t)\|^2\right] \dd t\to 0. 
\end{align}
Moreover, if $u_0\not\in H^1_0(\Om)$, then consider  an approximation $u^\e_0\in H^1_0(\Om)$ of the initial data such that $\|u^\e_0-u_0\|\leq \e$, for fixed $\varepsilon>0$, and let $\{(u_n^\e,v_n^\e)\}_{n\in \N}\subset (L^2(\Om))^2$ be the corresponding time-discrete solutions with $\{w^\e_n \}_{n\in \N}\subset \H$. Then, along any sequence of $(\e,\t)$ converging to $(0,0)$ one has
\begin{align}\label{eq:Conv-notH1}
    \sum_{n=1}^{N_\t} \int_{(n-1)\t}^{n\t} \left[\|u^\e_n - u(t)\|^2 + \|w_n^\e-\Phi(u(t))\|^2 + \|v^\e_n-v(t)\|^2\right] \dd t\to 0. 
\end{align} 
\end{theorem}

\noindent The proofs of Theorems \ref{thm: main time 1} and \ref{theo:time_convergence} are based on several lemmas.

\subsection{Proof of \Cref{thm: main time 1}: well-posedness, positivity, and boundedness}
\subsubsection{Existence and uniqueness}
 We first prove the existence and uniqueness of solutions of the system of equations \eqref{eq: time-discrete weak system}.

\begin{lemma}[Well-posedness for \eqref{eq: time-discrete weak system}] \label{lem: exist time-discrete}
    For $\t< 1\slash f_\Maxi$, there exists a unique weak solution of the time discretised system  \eqref{eq: time-discrete weak system}.
\end{lemma}

\begin{proof} 
\noindent \textbf{(Step 1) Existence of $u_n$:} 
\noindent As the equations are decoupled, we can first prove the existence of the solution $u_n$ of \eqref{eq: time-discrete weak 1}. To this end, we use arguments in \cite{Pop2011RegularizationSF}. We consider the function $\Psi \coloneqq \Phi^{-1}: \R^+\to [0,b)$ which satisfies
\begin{align}\label{eq:Psi}
    \Psi(0)=0, \quad \Psi'=\frac{1}{\Phi'\circ \Psi}\geq 0 
\end{align}
by \ref{ass: 1}.
Then, the energy $J\colon\H\to \R$, defined by
\begin{align}
    J(w)\coloneqq \int_\Om \left[h_{n-1}\int_0^w \Psi + \frac \t 2 |\del w|^2 -u_{n-1} w \right]
\end{align}
 is convex and coercive for $\t<1/f_\Maxi$. Hence, a minimizer $w_n\in \H$ of $J$ exists, and $u_n=\Psi(w_n)$ solves the corresponding Euler-Lagrange equation \eqref{eq: time-discrete weak 1}. Then using \ref{ass: 1}, for an arbitrary $\varepsilon\in (0,b)$, we have $0\leq \Psi\leq \Psi(\varepsilon)<\infty$ in $[0,\varepsilon]$ and $\Psi$ is Lipschitz in $(\varepsilon,b)$. Hence, by \eqref{eq:Psi}  it follows that, for a.e. $x \in \Om$ we have 
$
 0 \leq u_n^2(x) =  \Psi^2(w_n)(x) $, so $u_n$ is measurable. Integrating over $\Om$, one gets 
\begin{align*}
   0 \leq  \int_\Om u_n^2 &= \int_{\Om}\Psi^2(w_n) = \int_{\{0\leq w_n \leq \varepsilon\}}\Psi^2(w_n) + \int_{\{w_n >\varepsilon\}}\Psi^2(w_n) \\
    &\leq \Psi(\varepsilon)^2|\Om| + C \left(1+\norm{w_n}^2\right) \leq C + C \norm{\del w_n}_{L^2(\Om)}^2,
\end{align*}
for some constant $C>0$, where, in the last estimate, we used the Poincar\'e inequality \eqref{eq:Poincare}. Since $w_n \in \H$, this shows that the integral is finite, so $u_n\in L^2(\Om)$.

\textbf{(Step 2) Uniqueness of $u_n$:}
 Assume that for a given $v_{n-1}\in L^2(\Om)$, there are two solutions $u_n,\, \Tilde{u}_n\in L^2(\Om)$  of \eqref{eq: time-discrete weak 1} with $w_n=\Phi(u_n)$ and $\tilde w_n=\Phi(\Tilde{u}_n)$ in $H^1_0(\Om)$. For their difference we obtain
\begin{align}
   \frac{1}{\t} (h_{n-1}(u_n- \Tilde{u}_n),\f) + (\del (\Phi(u_n)-\Phi(\Tilde{u}_n)),\del \f) = 0 \qquad \forall \f\in H^1_0(\Om).\label{eq: Uniqueness eq1}
\end{align}
Note that $\phi = \posi{\Phi(u_n) - \Phi(\Tilde{u}_n)} \in \H$, see e.g.  \cite{cohn2013measure}. Choosing $\phi$ as a test function in \eqref{eq: Uniqueness eq1}  
leads to
\begin{equation}\label{eq: Uniqueness eq2}
    \frac{1}{\t}\brac{h_{n-1}\brac{u_n - \Tilde{u}_n}, \posi{\Phi(u_n) - \Phi(\Tilde{u}_n)}} + \norm{\nabla \posi{\Phi(u_n) - \Phi(\Tilde{u}_n)}}_{L^2(\Om)}^2 = 0.
\end{equation}
As $\Phi$ is an increasing function, we note that $\brac{u_n - \Tilde{u}_n}\posi{\Phi(u_n) - \Phi(\Tilde{u}_n)} \geq 0$. Hence, both terms in \eqref{eq: Uniqueness eq2} are non-negative and therefore, have to be equal to 0. We conclude that $\norm{\nabla \posi{\Phi(u_n) - \Phi(\Tilde{u}_n)}}_{L^2(\Om)}^2 = 0$, which results in $\norm{\posi{\Phi(u_n) - \Phi(\Tilde{u}_n)}}_{L^2(\Om)}^2 = 0$ by the Poincar\'e inequality \eqref{eq:Poincare}. This implies that $\Phi(\Tilde{u}_n) \geq \Phi(u_n)$ a.e. in $\Om$, but as $\Phi$ is an increasing function, we also find that $\Tilde{u}_n \geq u_n$ a.e. in $\Om$. Due to the symmetry in the arguments, it follows in the same way  that $u_n\geq \Tilde{u}_n$ a.e. which implies that $u_n=\Tilde{u}_n$ a.e. in $\Om$.
 The uniqueness of $u_n$ can also be shown via the $L^1$-contraction principle \cite{vazquez2007porous}. 

\textbf{(Step 3) Existence-uniqueness of $v_n$:} 
\noindent We now prove the existence and uniqueness for the solution $v_n$ of \eqref{eq: time-discrete weak 2}. In the PDE-ODE case, i.e. $\mu = 0$, we have an explicit expression for $v_n$, 
\begin{equation}
    v_n = v_{n-1} + \t g(u_n, v_{n-1}).
\end{equation}
Hence, the existence and uniqueness of $v_n$ follows from the existence and uniqueness of $u_n$. 

 In the PDE-PDE case, i.e. $\mu =1$, existence and uniqueness follows from the Lax-Milgram theorem \cite{evans10}. Indeed, the weak form can be rewritten as
\begin{equation}
    a(v_n, \eta) =  l(\eta) \quad \forall \eta \in \H,
\end{equation}
where the bilinear form $a\colon\H\times \H\to \R$ is given by $a(v_n, \eta) = (v_{n}, \eta) + \t \mu(D(u_n) \nabla v_{n}, \nabla \eta)$. It is  bounded and coercive since $0<D_\Mini\leq D\leq D_\Maxi<\infty$ by \ref{ass: 3}. Moreover, $l(\eta) = (\t g(v_{n-1},u_n) + v_{n-1}, \eta)$ is a bounded linear functional on $\H$. Consequently, there exists a unique solution $v_n\in \H$, which concludes the proof.
\end{proof}

\subsubsection{Positivity and boundedness in $L^\infty(\Om)$}

For the time-continuous biofilm models it was shown in \cite{HISSINKMULLER,mitra2023wellposedness} that the solutions $u$ and $v$ are non-negative, and that $u<1$. We aim to prove that these properties also hold for the time-discrete solutions. First, we derive bounds for $u_n$ in the general case, i.e. including porous medium type diffusion.

\begin{lemma}[Positivity and boundedness of $u_n$] \label{lem: un-bdd-pos}
     Let $u_{n-1}\in L^\infty(\Om)$ be positive a.e. in $\Om$. Then for $\t< 1\slash f_Maxi$, the solution $u_n\in L^2(\Om)$ of \eqref{eq: time-discrete weak 1} is positive and bounded a.e. in $\Om$. More precisely, we have
     \begin{equation}
    0\leq u_{n} \leq \sup \left\{ \frac{u_{n-1}}{1- \t f(v_{n-1})} \right\} \quad  \text{ a.e. in } \Om \label{eq: un bound}
\end{equation}
for all $1\leq n\leq N.$
\end{lemma}

\begin{proof}
    To prove that $u_n$ is bounded from above, we use the test function $\posi{\Phi(u_n) - \Phi(a)}$, for some $a \in \R^+$ in \eqref{eq: time-discrete weak 1}. Note that 
    $[\Phi(u_n) - \Phi(a)]_\pm\in H^1(\Omega)$ and
    $\left[\Phi(u_n) \mp \Phi(a)\right]_\pm$ = 0 if $u_n = 0$ as $\Phi(0) = 0$, and thus $\left[\Phi(u_n) \mp \Phi(a)\right]_\pm \in H^1_0(\Om)$. We find
\begin{align}
           &\int_\Om h_{n-1} \brac{u_n -a}\posi{\Phi(u_n) - \Phi(a)} + \int_\Om h_{n-1}\brac{a-\frac{1}{h_{n-1}}u_{n-1}}\posi{\Phi(u_n) - \Phi(a)} \nonumber\\
        &= -\t \int_{\Om} \nabla \Phi(u_n) \cdot \nabla \posi{\Phi(u_n) - \Phi(a)} = -\t \int_{\Om} \nabla \posi{\Phi(u_n) - \Phi(a)}^2 \leq 0. \label{eq:PositivityTest}
\end{align}
 Let $ \t < 1/f_\Maxi $ and $a \coloneqq \sup \frac{1}{h_{n-1}}u_{n-1}$, which implies that $a \geq 0$ as $h_{n-1}, u_{n-1} \geq 0$. 
 Then, the second term on the left hand side is positive. The first term is also positive as $\brac{u_n -a}\posi{\Phi(u_n) - \Phi(a)} \geq 0$ since $\Phi$ is increasing. 
 We conclude that the inequality in \eqref{eq:PositivityTest} must be an equality. This is only possible if $\posi{\Phi(u_n) - \Phi(a)} = 0$, and thus $\Phi(u_n) \leq \Phi(a)$. As $\Phi$ is an increasing function, this implies that $u_n \leq a$ and hence, 
\begin{equation}\label{eq: iterative upperbound un}
    u_n \leq a = \sup \left\{ \frac{1}{h_{n-1}}u_{n-1} \right\} = \sup \left\{ \frac{u_{n-1}}{1-\t f(v_{n-1})} \right\}.
\end{equation}
 We use the same arguments to prove that $u_n \geq 0$, but this time with $a = 0$ i.e. $\Phi(a)=\Phi(0)=0$. Using the test function $\phi = \nega{\Phi(u_n)}$ we conclude that $\nega{\Phi(u_n)} = 0$, and thus $u_n \geq 0$.
\end{proof}

An explicit upper bound for $u_n$ can be given in terms of the initial conditions and $f_\Maxi$, which 
we provide in the following result.

\begin{lemma}[Explicit upper bound $u_n$]\label{lem: expl upperbound un}
\noindent Let $u_0\in L^\infty(\Om)$ satisfy assumption \ref{ass: 4}. Then, for the sequence $\{(u_n,v_n)\}_{n=1}^N \subset L^2(\Om)^2$ solving \eqref{eq: time-discrete weak system}, one has

\begin{equation}\label{eq:LinfBound1}
     \|u_{n}\|_{L^{\infty}(\Om)} \leq \|u_{0}\|_{L^{\infty}(\Om)} \exp\brac{\frac{n\t f_\Maxi}{1-\t f_\Maxi}}.
\end{equation}
\end{lemma}

\begin{remark}[Upper bound as $\t\to 0$]    
Assuming that $u_{n}\to u(t)$ in $L^\infty(\Om)$ when $\t=t/n \to 0$,  the upper bound  \eqref{eq:LinfBound1} implies that 
\begin{equation}
    \|u(t)\|_{L^\infty(\Om)} \leq \|u_0\|_{L^\infty(\Om)}\exp\brac{tf_M}.
\end{equation}
\end{remark}
\begin{proof}
By \eqref{eq: iterative upperbound un} in the proof of \Cref{lem: un-bdd-pos}, we have using $\frac{1}{(1-x)}\leq \exp(\frac{x}{1-x})$ for $|x|<1$ that,
\begin{equation*}
\begin{split}
    \|u_{n}\|_{L^{\infty}(\Om)} &\leq  \sup \left\{ \frac{u_{n-1}}{1- \t f(v_{n-1})} \right\} \leq \frac{\|u_{n-1}\|_{L^{\infty}(\Om)}}{1- \t f_\Maxi}\leq \frac{\|u_{0}\|_{L^{\infty}(\Om)}}{\left(1- \t f_\Maxi\right)^n} \\
    &\leq \|u_{0}\|_{L^{\infty}(\Om)} \exp\brac{\frac{n\t f_\Maxi}{1-\t f_M}}.
\end{split}
\end{equation*}
\end{proof}

\noindent In the biofilm case, i.e. $b=1$, we can improve the upper bound. As shown e.g. in \cite{HISSINKMULLER, mitra2023wellposedness}, the solution $u$ of the time continuous system  is strictly less than 1, so we aim to prove this also for the approximate solutions $u_n$. \\

\begin{lemma}[Upperbound $u_n$ if $b=1$]\label{lem: un bdd1} Consider the biofilm case, i.e. $b = 1$, and let $u_0\in L^\infty(\Om)$ satisfy assumption \ref{ass: 4} and $\tau<1\slash f_\Maxi$. Then for the sequence $\{(u_n,v_n)\}_{n=1}^N \\ \subset L^2(\Om)^2$ solving \eqref{eq: time-discrete weak system}, one has
\begin{equation}
    0 \leq u_n \leq 1 - \delta\qquad \text{a.e. in}\ \Omega,
\end{equation}
for all $1 \leq n \leq N$, and some $\delta > 0$.
\end{lemma}

\begin{proof}
By \Cref{lem: un-bdd-pos} we have $0 \leq u_n \leq C$, for some constant $C>0$.
Let $\Tilde{\om} \in H^1_0(\Om) + \|\Phi(u_0)\|_{L^{\infty}(\Om)}$ be the solution of
\begin{equation} \label{eq: omtilde}
    \begin{split}
        \brac{\nabla \Tilde{\om}, \nabla \phi} = \left(C f_\Maxi, \phi \right) \quad \text{for all } \phi \in H^1_0(\Om).
    \end{split}
\end{equation}
 As $Cf_\Maxi \in \R^+_\ast$, by properties of the Poisson equation, we know that $\Tilde{\om} \in \Linf$. Further, since $\tilde{\om}$ is superharmonic, the maximum principle implies that $\Phi(u_0) \leq \Tilde{\om}$. We will prove that $\Phi(u_n) \leq \Tilde{\om}$ for all $1\leq n\leq N$ by induction. Assuming it holds for $n-1$, we subtract \eqref{eq: omtilde} from \eqref{eq: time-discrete weak 1} and multiply both sides by $\t$. We then choose the test function $\phi = \posi{\Phi(u_n) - \Tilde{\om}} \in \H$ to find
\begin{equation}
    \begin{split}
        &\brac{u_n - \Phi^{-1}(\Tilde{\om}) + \Phi^{-1}(\Tilde{\om}) - u_{n-1}, \posi{\Phi(u_n) - \Tilde{\om}}} + \t \brac{\nabla\brac{\Phi(u_n) - \Tilde{\om}}, \nabla \posi{\Phi(u_n) - \Tilde{\om}}} \\
        &+ \t \brac{f_M C - f(v_{n-1})u_n, \posi{\Phi(u_n) - \Tilde{\om}}} = 0.
    \end{split}
\end{equation}
 By the induction hypothesis, we have $\Phi^{-1}(\Tilde{\om}) - u_{n-1}\geq 0$, while ${\brac{u_n - \Phi^{-1}(\Tilde{\om})}\posi{\Phi(u_n) - \Tilde{\om}}}$ ${\geq 0}$ as $\Phi$ is an increasing function. The other terms are also positive as \[\t \brac{\nabla\brac{\Phi(u_n) - \Tilde{\om}}, \nabla \posi{\Phi(u_n) - \Tilde{\om}}} = \t \norm{\nabla \posi{\Phi(u_n) - \Tilde{\om}}}^2 \text{ and } f_\Maxi C - f(v_{n-1})u_n \geq 0\] by definition of $f_\Maxi$ and $C$. Hence, the Poincare inequality implies that
\begin{equation}
    \norm{\posi{\Phi(u_n) - \Tilde{\om}}}_{L^2(\Om)}^2 \leq C_{\Om} \norm{\nabla\posi{\Phi(u_n) - \Tilde{\om}}}_{\Lom{2}}^2 = 0,
\end{equation}
and thus $\Phi(u_n) \leq \Tilde{\om}$. To conclude the proof, we recall that $\Tilde{\om}$ is bounded and hence, 
\begin{equation}
    \begin{split}
        0 \leq u_n \leq \Phi^{-1}(\Tilde{\om}) = 1 - \delta, \quad \delta > 0.
    \end{split}
\end{equation}
\end{proof}

\begin{remark} [Effective Lipschitz continuity of $\Phi$]
By \Cref{lem: un bdd1}, in the biofilm case, i.e. $b=1$, we can effectively restrict the domain of $\Phi$ to $[0,1-\delta] \subset [0,1)$. Within this interval, $\Phi'$ is Lipschitz continuous as stated in assumption \ref{ass: 1}.
\end{remark}

 We have shown that $u_n \leq  1-\delta$ for some $\delta>0$, but we aim to derive an explicit bound. Such a bound will be useful in Section \ref{sec: linear} when we propose the linearisation scheme and is  provided in the following lemma. 

\begin{lemma}[Explicit upper bound $u_n$ if $b=1$]\label{lem: max_un}
Consider the biofilm case, i.e. $b = 1$, and let $u_0\in L^\infty(\Om)$ satisfy assumption \ref{ass: 4} and $\tau<1\slash f_\Maxi$. Then for the sequence $\{(u_n,v_n)\}_{n=1}^N \subset L^2(\Om)^2$ solving \eqref{eq: time-discrete weak system}, one has
\begin{equation}
    0 \leq u_n \leq \Phi^{-1}(\Tilde{C}) < 1\qquad \text{a.e. in}\ \Omega,
\end{equation}
for all $1\leq n \leq N$, where
\begin{equation}\label{eq:LinfBound2}
\Tilde{C} = \infnorm{\Phi(u_0)} + \frac{{\rm diam}(\Om)^2}{2d} f_\Maxi,
\end{equation}
and $d$ is the spatial dimension of $\Om \subset \R^d$.    
\end{lemma}

\begin{proof}
Let $\tilde{\om}$ be the solution of \Cref{eq: omtilde} as in the proof of \Cref{lem: un bdd1}. Then, $\tilde{\om}$ is a classical solution to the elliptic problem
\begin{equation}
    \begin{split}
        \begin{cases}
            -\Delta \Tilde{\om} = f_\Maxi & \text{on $\Omega$}, \\
            \Tilde{\om} = \infnorm{\Phi(u_0)} & \text{on $\p \Om$},
        \end{cases}
    \end{split}
\end{equation}
as we can take $C = 1$ by \Cref{lem: un bdd1}.

 As $f_\Maxi > 0$, $\Tilde{\om}$ is superharmonic, and we conclude that $\Tilde{\om} \geq \infnorm{\Phi(u_0)} \geq 0$
by the maximum principle. We then define the function $w = f_\Maxi\norm{x - \Bar{x}}^2 / (2d) \geq 0$, where $d$ is the spatial dimension and $\Bar{x}$ is given by $\Bar{x}_i = |\Om|^{-1} \int_{\Om}x_i$. It is straightforward to see that $\Delta w = f_\Maxi$.

 If we consider $z = \Tilde{\om} + w$, it satisfies $-\Delta z = - \Delta \Tilde{\om} - \Delta w  = 0$. Hence, by the maximum principle it follows that
\begin{equation}
    \begin{split}
         \infnorm{\Tilde{\om}} \leq \infnorm{z} \leq \norm{z}_{L^{\infty}(\p \Om)} &= \norm{\Tilde{\om}}_{L^{\infty}(\p \Om)} + \frac{\norm{x-\Bar{x}}^{2}_{L^{\infty}(\p \Om)}}{2d}f_\Maxi\\
         &= \infnorm{\Phi(u_0)} + \frac{{\rm diam}(\Om)^2}{2d} f_\Maxi,
    \end{split}
\end{equation}
where we used that $w, \Tilde{w},z \geq 0$, which implies that $0 \leq \Tilde{w} \leq z$. This concludes the proof.
\end{proof}

 We now show that $v_n$ is positive, similarly as in \Cref{lem: un-bdd-pos} for $u_n$, for both cases  $b=1$ and $b=\infty$.

\begin{lemma}[Positivity of $v_n$] \label{lem: vn-pos}
Let the assumptions of Lemma 3.4 hold.
Let $v_{n-1}\in L^2(\Omega)$ be positive a.e. in $\Om$. Then for $\t < 1 \slash g_\Maxi$ the solution $v_n\in L^2(\Omega)$, $\mu v_n\in \H$ of \eqref{eq: time-discrete weak 2} is positive a.e. in $\Om$.
\end{lemma}

\begin{proof}
By \Cref{lem: un-bdd-pos,lem: expl upperbound un,lem: un bdd1}, $u_{n}$ is positive and bounded in $[0,b)$ implying that $g(u,\cdot)$ is well-defined. For the positivity of $v$, first observe that $G(u,t):=t +\t g(u,t)$ is an increasing function in $t$ for $\t<1/g_\Maxi$. For $\mu=1$, inserting $\eta = [v_n]_-$ in 
\eqref{eq: time-discrete weak 2} gives
\begin{align*}
   &(v_{n}, [v_n]_-) + \t \mu(D(u_n) \nabla v_{n}, \nabla [v_n]_-) = \t(g(u_n, v_{n-1}), [v_n]_-) + (v_{n-1},[v_n]_-)\\
 = & \ \t(g(u_n, v_{n-1})-g(u_n, 0), [v_n]_-) + \t(g(u_n, 0), [v_n]_-) + (v_{n-1},[v_n]_-) \\
 =&\ \t (G(u_n,v_{n-1})-G(u_n,0), [v_n]_-) + \t(g(u_n, 0), [v_n]_-) \leq 0,
\end{align*}
since $G(u_n,v_{n-1})\geq G(u_n,0)$ due to $v_{n-1}$ being positive, and $g(u_n,0)\geq 0$ from \ref{ass: 2}. Using a similar test function in the case $\mu=0$, we conclude that 
\begin{equation}
   \int_{\Omega}[v_{n}]_{-}^{2} + \t \mu \int_{\Omega}D(u_n) \nabla [v_{n}]_{-}^{2} \leq 0
\end{equation}
from which we conclude that $[v_n]_-=0$ a.e. in $\Om$, or in other words, $v_n \geq 0$.
\end{proof}

\noindent We now have all the necessary results to prove \Cref{thm: main time 1}. 

\begin{proof}[Proof of \Cref{thm: main time 1}]
The existence and uniqueness of weak solutions of \eqref{eq: time-discrete weak system} is provided by \Cref{lem: exist time-discrete}. The positivity and boundedness of $u_n$ are proven in \Cref{lem: un-bdd-pos}, while the explicit bounds are given in \Cref{lem: expl upperbound un} and \Cref{lem: max_un}. Finally, the positivity of $v_n$ is the result of \Cref{lem: vn-pos}.
\end{proof}

\subsection{Proof of \Cref{theo:time_convergence}: convergence of the time-discrete solutions}
Here Rothe's method 
is used to prove the convergence of the time-discrete solutions $\{(u_n,v_n)\}_{n\in \N}\subset (L^2(\Om))^2 $ of \eqref{eq: time-discrete weak system}. For a time-step size $\t>0$ with time-steps $t_n:=n\t$ (recall that $T= N_\t \t$ is fixed), and a sequence $\{z_n\}_{n\in \N}\subset L^2(\Om)$, we construct the piece-wise constant and affine  time-interpolations $\hat{z}_\t,\,\bar{z}_\t\in L^2(\Om\times [0,T])$ as 
\begin{align}\label{eq:time_interpolants}
\hat{z}_{\t}(t):= z_{n}, \quad \bar{z}_{\t}(t):=z_{n-1}+ \frac{t-t_{n-1}}{\t}(z_n-z_{n-1})\quad \text{ if } t\in (t_{n-1},t_n] \text{ for some } n\in \N. 
\end{align}

\subsubsection{Uniform boundedness of the interpolates in Bochner spaces}

\begin{lemma}[Uniform boundedness of $\hat{w}_\t$, $\Bar{w}_\t$, $\Bar{v}_\t$ with respect to $\t$]\label{lem:unif bdd w}
For a time-step size $\t>0$, let $\{(u_n,w_n,v_n)\}_{n\in \N}$ satisfy the assumptions in \Cref{theo:time_convergence}, and let $\{\hat{w}_\t\}$, $\{\hat{v}_\t\}$, $\{\Bar{w}_\t\}$, $\{\Bar{v}_\t\}$ be the piecewise constant, respectively piecewise linear time-interpolations introduced in \eqref{eq:time_interpolants}. Then, there exists a constant $\bar{C}>0$ independent of $\t>0$ such that for both, $z=\Bar{w}_\t,$ and $z=\hat{w}_\t$, it holds
\begin{align}\label{eq:Uni_bnd_w}
& \sup_{0\leq t\leq T} \|z(t)\|^2_{L^\infty(\Om)}  +\int_0^T \|\del z\|^2\leq \bar{C},\\
&  \sup_{0\leq t\leq T} \|\del z(t)\| + \int_0^T \|\p_t \Bar{w}_\t\|^2 \leq \bar{C} [1 +  \|\del \Phi(u_0)\|^2]. \label{eq:Uni_bnd_w_H1}
\end{align}
Additionally, for $\mu\in \{0,1\}$  it holds
\begin{align}\label{eq:Uni_bnd_v}
& \sup_{0\leq t\leq T} \|\Bar{v}_\t(t)\|^2 + \int_0^T \|\p_t \Bar{v}_\t\|^2_{H^{-1}(\Om)}   +(1-\mu)\sup_{0\leq t\leq T} \|\p_t \Bar{v}_\t(t)\|^2 +\mu \int_0^T \|\del\Bar{v}_\t\|^2\leq \bar{C}.
\end{align}    
\end{lemma}

\begin{proof}
    Observe that \Cref{thm: main time 1}, specially \eqref{eq:uup} directly yields 
\begin{align*}
   &\sup_{0\leq t\leq T} \|\hat{w}_\t(t)\|_{L^\infty(\Om)}\leq  \sup_{1\leq n\leq N_\t} \|w_n\|_{L^\infty(\Om)}\overset{\eqref{eq:uup}}\leq \Phi(\Breve{u})<\infty.
\end{align*}
Similarly, $\|\bar{w}_\t(t)\|_{L^\infty(\Om)}<\infty$ since $\bar{w}_\t(t)$ is a convex combination of $\{w_n\}_{n\in\N}$. The other estimates follow closely the Rothe method, see e.g. \cite{mitra2023wellposedness} for an identical context, or \cite{mitra2021existence}.

\textbf{(Step 1) Bound \eqref{eq:Uni_bnd_v}:} Inserting $\eta=v_n$ in \eqref{eq: time-discrete weak 2} one has
\begin{align}
    \frac{1}{\t}(v_n-v_{n-1},v_n) + \mu (D(u_n)\del v_n,\del v_n)= (g(u_n,v_{n-1}),v_n).
\end{align}
To rewrite the first term we use the identity $a(a-b)=\frac{1}{2}[a^2-b^2 +(a-b)^2]$,
\begin{subequations}
    \begin{align}
     &\frac{1}{\t}(v_n-v_{n-1},v_n)   = \frac{1}{2\t}[\|v_n\|^2 - \|v_{n-1}\|^2 + \|v_n-v_{n-1}\|^2],
     \end{align}
and for the second term, \ref{ass: 3} implies that
     \begin{align}
     & \mu (D(u_n)\del v_n,\del v_n)\geq \mu D_\Mini \|\del v_n\|^2.
     \end{align}
For the third term, notice that $|g(u_n,v_{n-1})|\leq  |g(u_n,v_{n-1})-g(u_n,0)|+ |g(u_n,0)|\leq C[1+ |v_{n-1}|]$ for some constant $C>0$, which follows from the Lipschitz continuity of $g$ in \ref{ass: 2} and \eqref{eq:uup}. Then, one has
     \begin{align}
     &  (g(u_n,v_{n-1}),v_n)=(g(u_n,v_{n-1}),v_n-v_{n-1}) + (g(u_n,v_{n-1}),v_{n-1})\\
     &\overset{\ref{ass: 2},\eqref{eq: Young}}\leq C[1+ \|v_{n-1}\|^2] + \|v_n-v_{n-1}\|^2,
    \end{align}
\end{subequations}
and summing up the estimates above from $n=1$ to $n=N_\t$, we obtain
\begin{align*}
    \|v_{N_\t}\|^2+ (1-2\t)\sum_{n=1}^{N_\t}\|v_n-v_{n-1}\|^2 + 2\mu D_\Mini \sum_{n=1}^{N_\t}\|\del v_n\|^2\t \leq \|v_0\|^2 + 2C \sum_{n=1}^{N_\t} [1+ \|v_{n-1}\|^2]\t.
\end{align*}
For $\t<1/2$, applying the discrete Gronwall lemma \eqref{eq: discr Gronwall} to the above inequality reveals that $\|v_n\|$ is uniformly bounded with respect to $\t$ provided $1\leq n \leq N_\t$. Substituting this back into the above inequality, one obtains
\begin{align}\label{eq:v-disc-bound}
    \|v_{N_\t}\|^2+ \sum_{n=1}^{N_\t}\|v_n-v_{n-1}\|^2 + \mu D_\Mini\sum_{n=1}^{N_\t}\|\del v_n\|^2\t  \leq C.
\end{align}

Observe that $\hat{v}_\t(t)=v_n$ for $t\in (t_{n-1},t_n]$, and $\bar{v}_\t$ is a convex combination of $\{v_n\}_{n\in \N}$. Hence, the above inequality implies that $\|\hat{v}_\t(t)\|$ and 
 $\|\Bar{v}_\t(t)\|$ are uniformly bounded with respect to $\t$. Likewise, $\int_0^T \|\del \hat{v}_\t\|^2 =  \sum_{n=1}^{N_\t}\|\del v_n\|^2\t$ which is uniformly bounded due to \eqref{eq:v-disc-bound} if $\mu=1$, and the same also holds for $\int_0^T \|\del \bar{v}_\t\|^2$. Observe that for $t\in (t_{n-1},t_n]$,
 \begin{subequations}\label{eq:dtv-bound}
 \begin{align}
    &\|\p_t\bar{v}_\t(t)\|_{H^{-1}(\Om)} :=\sup_{\substack{\eta\in H^1_0(\Om)\\\|\del \eta\|=1
    }}\left\langle \frac{1}{\t}(v_n-v_{n-1}),\eta\right \rangle \nonumber\\
    &\qquad\overset{\eqref{eq: time-discrete weak 1}}=\sup_{\substack{\eta\in H^1_0(\Om)\\\|\del \eta\|=1
    }} \left[-\mu(D(u_n)\del v_n,\del \eta) + (g(u_n,v_{n-1}),\eta)\right]\nonumber\\ &\qquad\overset{\eqref{eq:Poincare}, \ref{ass: 2}}\leq \mu 
 D_\Maxi\|\del v_n\| + C_\Om\|g(u_n,v_{n-1})\|\overset{\eqref{eq:v-disc-bound},\eqref{eq:uup}}\leq \mu D_\Maxi\|\del \hat{v}_\t\| + C.
 \end{align}
This implies that $\int_0^T\|\p_t \bar{v}_\t(t)\|^2_{H^{-1}(\Om)}$ is uniformly bounded with respect to $\t$ since $\int_0^T \|\del \hat{v}_\t\|^2$ is.
If in addition $\mu=0$, then 
\begin{align}
    \|\p_t \bar{v}_\t(t)\|=\|g(u_n,v_{n-1})\|\overset{\eqref{eq:v-disc-bound},\eqref{eq:uup}}\leq C.
\end{align}
 \end{subequations}
Combining \eqref{eq:v-disc-bound} and \eqref{eq:dtv-bound} we obtain \eqref{eq:Uni_bnd_v}.

\textbf{(Step 2) Bounds \eqref{eq:Uni_bnd_w}--\eqref{eq:Uni_bnd_w_H1}:} Proving \eqref{eq:Uni_bnd_w}, requires taking $\phi=w_n$ as a test function in \eqref{eq: time-discrete weak 1}.
 The arguments are identical to Step 2 in the proof of Lemma 4.3 in \cite{mitra2023wellposedness} and hence, will be omitted for the sake of brevity. For obtaining \eqref{eq:Uni_bnd_w_H1}, we insert $\phi=w_n-w_{n-1}=\Phi(u_n)-\Phi(u_{n-1})$ in \eqref{eq: time-discrete weak 1} to get
\begin{align}\label{eq:proof unif bound}
    \left(\frac{1}{\t}(u_n - u_{n-1}),\Phi(u_n)-\Phi(u_{n-1})\right) + (\nabla w_n, \nabla (w_n-w_{n-1})) = (f(v_{n-1})u_n, w_n-w_{n-1}).
\end{align}
 Noting that  $\p_t \Bar{w}_\t=(\Phi(u_n)-\Phi(u_{n-1}))/\t$ for $t\in (t_{n-1},t_n]$ and $L_\Phi:=\sup_{u\in [0,\Breve{u}]}\{\Phi'(u)\}<\infty$ from \eqref{eq:uup}, the first term in \eqref{eq:proof unif bound} is estimated as
\begin{subequations}
    \begin{align}
    \left(\frac{1}{\t}(u_n - u_{n-1}),\Phi(u_n)-\Phi(u_{n-1})\right) &\overset{\eqref{eq:uup}}\geq \frac{\t}{\sup\limits_{u\in [0,\Breve{u}]} \Phi'(u)}\left\|\frac{\Phi(u_n)-\Phi(u_{n-1})}{\t}\right\|^2\nonumber\\
    &= \frac{\t}{L_\Phi}  \|\p_t \Bar{w}_\t\|^2.
\end{align}
Using the identity $a(a-b)=\frac{1}{2}[a^2-b^2 +(a-b)^2]$, the second-term is estimated as 
\begin{align}
    (\nabla w_n, \nabla (w_n-w_{n-1}))= &\frac{1}{2}[\|\del w_n\|^2-\|\del w_{n-1}\|^2 + \|\del (w_n-w_{n-1})\|^2],
\end{align}
 Similarly as in Step 1, using that $\|u_n\|_{L^\infty(\Om)}<C$ by \eqref{eq:uup}, $\|v_n\|<C$ by \eqref{eq:v-disc-bound}, and that $f$ is a Lipschitz function by \ref{ass: 2}, we have that $\|f(v_{n-1})u_n\|<C$. Then, the final term is estimated as 
\begin{align}
    (f(v_{n-1})u_n, w_n-w_{n-1}))&\overset{\eqref{eq: Young}}\leq \frac{L_\Phi}{2} \|f(v_{n-1})u_n\|^2\t + \frac{\t}{2L_\Phi}\left\|\frac{w_n-w_{n-1}}{\t}\right\|^2\nonumber\\
    &\leq C\t + \frac{\t}{2L_\Phi} \|\p_t \Bar{w}_\t\|^2
\end{align}
\end{subequations}
Combining the above estimates and summing from $n=1$ to $n=N_\t$ we get 
\begin{align}\label{eq:w-disc-bound}
   \frac{1}{L_\Phi} \sum_{n=1}^{N_\t}\|\p_t \Bar{w}_\t\|^2\t + \|\del w_{N_\t}\|^2+ \sum_{n=1}^{N_\t}\|\del (w_n-w_{n-1})\|^2 \leq CT + \|\del \Phi(u_0)\|^2.
\end{align}
Since $\hat{w}_\t(t)=w_n$ for $t\in (t_{n-1},t_n]$, and $\bar{w}_\t$ is a convex combination of $w_n$ and $w_{n-1}$, similarly as in Step 1, we conclude that $\|\del \hat{w}_
t(t)\|$ and $\|\del \bar{w}_\t(t)\|$ are bounded for all $t\in [0,T]$. Finally, observing that $\int_0^T\|\p_t \Bar{w}_\t\|^2=\sum_{n=1}^{N_\t}\|\p_t \Bar{w}_\t\|^2\t$, we have \eqref{eq:Uni_bnd_w_H1}.
\end{proof}

\subsubsection{Convergence to the time-continuous solution if $u_0\in H^1_0(\Om)$}
We first prove the following result which will be used frequently:
\begin{lemma}[An important convergence result]\label{lemma:imprtnt_conv}
Let $\psi\in C^1(\R^+)$ be strictly increasing, convex in $[0,\e_0)$ for some $\e_0>0$, and 
assume that for $\psi_\Mini := \inf_{[\e_0,\infty)}\psi'$ one has $\psi_\Mini > 0$. For a measurable set $\om\subset\R^d$, let $\{\f_n\}_{n\in \N}\subset L^1(\om)$ be a sequence of non-negative functions such that $\|\psi(\f_n)-\psi(\f)\|_{L^1(\om)}\to 0$ for a fixed (non-negative) $\f\in L^1(\om)$. Then, $\|\f_n-\f \|_{L^1(\om)}\to 0$ as $n\to \infty$. 
\end{lemma}
\begin{proof}
Let $\bar{\Psi}\in C(\R^+)$ be defined as
\begin{equation*}
    \bar{\Psi}(\f) = \begin{cases} \psi(\f)-\psi(0) &\text{ for } \f\in [0,\varepsilon_0),\\ 
    \psi(\e_0)-\psi(0) + \psi'(\e_0) (\f-\e_0) &\text{ for } \f\geq \e_0.
    \end{cases}
\end{equation*}
It is straightforward to verify that $\bar{\Psi}$ is convex, strictly increasing, $\bar{\Psi}(0)=0$, and
\begin{align}\label{eq:Psi-psi-ineq}
    |\bar{\Psi}(\f_1)-\bar{\Psi}(\f_2)|\leq (\psi'(\e_0)/\psi_\Mini)|\psi(\f_1)-\psi(\f_2)| \text{ for all } \f_{1/2}\geq 0.
\end{align}
The inequality above follows from considering separately the cases $\f_{1/2}\leq \e_0$ which gives $|\bar{\Psi}(\f_1)-\bar{\Psi}(\f_2)|= |\psi(\f_1)-\psi(\f_2)|$; $\f_{1/2}\geq \e_0$ which gives $|\bar{\Psi}(\f_1)-\bar{\Psi}(\f_2)|= \psi'(\e_0) |\f_1-\f_2|\leq (\psi'(\e_0)/\psi_\Mini) |\psi(\f_1)-\psi(\f_2)|$; and $\f_1$, $\f_2$ being on different sides of $\e_0$ which gives also \eqref{eq:Psi-psi-ineq} by combining the estimates for the other two cases. Moreover,  using the super-additivity property \eqref{eq: superadd}  one has for $\f_n>\f$ that
$\bar{\Psi}(\f_n-\f)\leq \bar{\Psi}(\f_n)-\bar{\Psi}(\f)$, and by symmetry, we conclude that $\bar{\Psi}(|\f_n-\f|)\leq |\bar{\Psi}(\f_n)-\bar{\Psi}(\f)|$. Consequently,
\begin{equation*}
    \begin{split}
        \bar{\Psi} \brac{\frac{1}{|\om|} \int_{\om}|\f_n-\f|} &\overset{\eqref{eq: jensen}}\leq \frac{1}{|\om|} \int_{\om} \bar{\Psi}\brac{|\f_n - \f|} \leq\frac{1}{|\om|} \int_{\om} |\bar{\Psi}(\f_n) - \bar{\Psi}(\f)|\\
        &\leq  C \|\psi(\f_n) - \psi(\f)\|_{L^1(\om)}\longrightarrow 0.
    \end{split}
\end{equation*}
Since $\bar{\Psi}$ is strictly increasing, it follows that $\|\f_n-\f \|_{L^1(\omega)}\to 0$.
\end{proof}
The above result has previously been used in Lemma 3.3 of \cite{mitra2023wellposedness} to prove strong convergence of solutions, see also \cite{HISSINKMULLER}. Here, we use it in a similar way.

\begin{proof}[\underline{Proof of \eqref{eq:Conv-H1} in \Cref{theo:time_convergence}}]
Observe that $\bar{w}_\t \in H^1(Q)$ is uniformly bounded with respect to $\t$ if $u_0\in H^1_0(\Om)$ by Lemma \ref{lem:unif bdd w} since $\Phi(u_0)\in H^1_0(\Om)$ in this case. Hence, by the compact embedding  $H^1(Q)\hookrightarrow L^2(Q)$, there exists $w\in H^1(Q)$ such that along a sub-sequence of $\t$ converging to $0$,
\begin{subequations}
    \begin{align}
        &\bar{w}_\t\rightharpoonup w\ \text{ weakly in } H^1(Q),\\
        &\bar{w}_\t\longrightarrow w\ \text{ strongly in } L^2(Q).\label{eq:bar_w_conv}
    \end{align}
Define $u:=\Phi^{-1}(w)$ which is bounded in $[0,\Breve{u}]$ a.e. in $\Om$ for all $t>0$ due to \eqref{eq:uup}. We will  prove that
\begin{align}
        &\hat{w}_\t\longrightarrow w\ \text{ strongly in } L^2(Q)\label{eq:hat_w_conv}\\
        &\hat{u}_\t\longrightarrow u\  \text{ strongly in } L^2(Q).\label{eq:hat_u_conv}
\end{align}
\end{subequations}
The  convergence \eqref{eq:hat_w_conv} follows from \eqref{eq:time_interpolants} and \eqref{eq:bar_w_conv} since
\begin{align*}
    \int_0^T\|\hat{w}_\t-\hat{w}_\t\|^2&\overset{\eqref{eq:time_interpolants}}=\sum_{n=0}^N \int_{t_{n-1}}^{t_n}\left (\frac{t-t_{n-1}}{\t}\right)^2\|w_n-w_{n-1}\|^2= \frac{1}{3}\sum_{n=0}^N \|w_n-w_{n-1}\|^2 \t \nonumber\\
    &\overset{\eqref{eq:Poincare}}\leq\frac{C_\Om\t}{3}\sum_{n=0}^N \|\del(w_n-w_{n-1})\|^2\overset{\eqref{eq:w-disc-bound}}\leq  C\t\longrightarrow 0.
\end{align*}
To show \eqref{eq:hat_u_conv}, noting that $\Phi(\hat{u}_\t)=\hat{w}_\t$, we have
\[ \|\Phi(\hat{u}_\t)-\Phi(u)\|_{L^2(Q)}=\|\hat{w}_\t-w\|^2_{L^2(Q)}\longrightarrow 0,\]
which also implies that $\|\Phi(\hat{u}_\t)- \Phi(u)\|_{L^1(Q)}\to 0$. Hence, using \Cref{lemma:imprtnt_conv} with $\psi=\Phi$ gives that $\|\hat{u}_\t- u\|_{L^1(Q)}\to 0$ and since both $\hat{u}_\t,u\in L^\infty(Q)$, we have \eqref{eq:hat_u_conv}.

For the convergence of $v$, note that if $\mu=1$ then \eqref{eq:Uni_bnd_v} implies that $\bar{v}_\t\in H^1(0,T;H^{-1}(\Om))\cap L^2(0,T;H^1_0(\Om))=:{\cal W}$ is uniformly bounded with respect to $\t$. The space ${\cal W}$ is compactly embedded into $L^2(Q)$ and continuously into $C([0,T];L^2(\Om))$ (Aubin-Lions lemma). Hence, for $\mu=1$, there exists $v\in {\cal W}\subset C([0,T];L^2(\Om)) $ such that 
\begin{subequations}
\begin{align}
        &\bar{v}_\t\longrightarrow v\  \text{ strongly in } L^2(Q),\label{eq:bar_v_conv}\\
        &\hat{v}_\t\longrightarrow v\  \text{ strongly in } L^2(Q),\label{eq:hat_v_conv}
\end{align}
\end{subequations}
For $\mu=0$, let $v\in C([0,T];L^2(\Om))$ be the solution of $\p_t v=g(u,v)$ with $v(0)=v_0$. Then, 
\begin{align*}
    \frac{1}{2}\|(\bar{v}_\t- v)(T)\|^2&=\int_0^T (\bar{v}_\t- v,\p_t(\bar{v}_\t- v))\leq \|\bar{v}_\t- v\|_{L^2(Q)} \|\p_t(\bar{v}_\t- v)\|_{L^2(Q)}\\
    &\overset{\eqref{eq:Uni_bnd_v}} \leq C \|\p_t(\bar{v}_\t- v)\|_{L^2(Q)}.
\end{align*}
Using that $\p_t \bar{v}_\t(t)=(v_n-v_{n-1})/\t=g(u_n,v_{n-1})=g(\hat{u}_\t(t),v_{n-1})$ for $t\in (t_{n-1},t_n]$, one further estimates
\begin{align*}
    &\|\p_t(\bar{v}_\t- v)\|^2_{L^2(Q)}=\sum_{n=1}^{N_\t}\int_{t_{n-1}}^{t_n} \|g(\hat{u}_\t,v_{n-1})-g(u,v)\|^2 \\&\overset{\ref{ass: 2}}
\leq C\sum_{n=1}^{N_\t}\int_{t_{n-1}}^{t_n} \left[ \|\hat{u}_\t-u\|^2 + \|v_{n-1}-v\|^2\right]\nonumber\\
&\leq C\left(\|\hat{u}_\t-u\|^2_{L^2(Q)} + \int_0^T \|\bar{v}_\t-v\|^2 + \sum_{n=1}^{N_\t}\int_{t_{n-1}}^{t_n} \|v_{n-1}-\bar{v}_\t\|^2 \right).
\end{align*}
Note that $\|\hat{u}_\t-u\|^2_{L^2(Q)}\to 0$ from \eqref{eq:hat_u_conv}, and 
\begin{equation*}
    \begin{split}
        \sum_{n=1}^{N_\t}\int_{t_{n-1}}^{t_n}  \|v_{n-1}-\bar{v}_\t\|^2&\overset{\eqref{eq:time_interpolants}}= \sum_{n=1}^{N_\t}\int_{t_{n-1}}^{t_n}  \left(\frac{t-t_{n-1}}{\t}\right)^2\|v_n-v_{n-1}\|^2\\
        &= \frac{\t}{3}\sum_{n=1}^{N_\t} \|v_n-v_{n-1}\|^2\overset{\eqref{eq:v-disc-bound}}\leq C\t \longrightarrow 0,
    \end{split}
\end{equation*}
Hence, applying Gronwall's lemma we get that $\|(\bar{v}_\t- v)(T)\|\to 0$ which proves the strong convergence result in \eqref{eq:bar_v_conv}. The convergence of $\hat{v}_\t$ in  \eqref{eq:hat_v_conv} follows from \eqref{eq:v-disc-bound} similar to \eqref{eq:hat_w_conv}.

It is straightforward to show that $(u,v)$ is indeed a weak solution of \eqref{eq: main}, a detailed proof can be found in Theorem 3.1 in \cite{mitra2021existence}. Since $(\hat{u}_\t, \hat{w}_\t, \hat{w}_\t)$ is  bounded uniformly componentwise in $L^2(Q)$ for $\t$ small, and every converging subsequence of it converges to the unique limit $(u,v,w)$  weakly solving \eqref{eq: main}, along every sequence of $\t\to 0$  this limit is obtained.
\end{proof}

\subsubsection{Convergence to the time-continuous solution if $u_0\not\in H^1_0(\Om)$}

For less regular initial data we need to use the $L^1$-contraction principle, see \cite{otto1996l1} for the general idea, and \cite{HISSINKMULLER} for a proof for this specific case.
\begin{lemma}[$L^1$-contraction principle]\label{lemma:L1}
    Let $(u_1,v_1)$ and $(u_2,v_2)$  be the weak solutions of \eqref{eq: main} 
    corresponding to the initial data $u_1(0)=u_{1,0}$ and $u_2(0)=u_{2,0}$ and let $u_{1,0}, u_{2,0}$ satisfy  \ref{ass: 4}. Then, for any $t>0$
    \begin{align}
        \|(u_1-u_2)(t)\|_{L^1(\Om)}\leq \|u_{1,0}-u_{2,0}\|_{L^1(\Om)} + \int_0^t \|f(v_1)u_1 - f(v_2)u_2\|_{L^1(\Om)}.
    \end{align}
\end{lemma}

\begin{lemma}[Convergence of the continuous solutions as $\e\to 0$]\label{lemma:eps-conv}
Let $(u,v)$ and $(u^\e,v^\e)$ be the weak solutions of \eqref{eq: main} corresponding to the initial conditions $u(0)=u_{0}$ and $u^\e(0)=u_{0}^\e$, where $u_{0}^\e$ is as in \Cref{theo:time_convergence}, and let $w=\Phi(u), w^\e=\Phi(u^\e)$. Then, for any $t>0$,
along any sequence of $\e$ converging to 0 we have
\begin{align}
         \|(u^\e-u)(t)\|_{L^2(\Om)} + \|(w^\e-w)(t)\|_{L^2(\Om)}+ \|v^\e-v\|_{L^2(Q)}\longrightarrow 0.
\end{align}
\end{lemma}

\begin{proof}
Observe that the uniform bound in \eqref{eq:Uni_bnd_v}  holds also for $z=v^\e$ with the constant $\bar{C}$ independent of $\e$. Hence, similar to \eqref{eq:bar_v_conv}, along a subsequence of $\e\to 0$, one has $\|v^\e-v\|_{L^2(Q)}\to 0$. Moreover, noting that $0\leq u(t),\, u^\e(t)\leq \Breve{u}<C$ a.e. in $\Om$ due to \eqref{eq:Uni_bnd_w}, one has by \Cref{lemma:L1} that 
\begin{align}
 \|(u^\e-u)(t)\|_{L^1(\Om)}&\leq \|u^\e_0-u_0\|_{L^1(\Om)} + \int_0^t \|f(v^\e)u^\e - f(v)u\|_{L^1(\Om)}\nonumber\\
 &\hspace{-2em}\leq |\Om|^{\frac{1}{2}}\|u^\e_0-u_0\| + \int_0^t \|(f(v^\e) - f(v))u^\e\|_{L^1(\Om)} + \int_0^t \|f(v)(u^\e-u)\|_{L^1(\Om)}\nonumber\\
  &\hspace{-2em}\overset{\ref{ass: 2}}\leq \e|\Om|^{\frac{1}{2}} \overset{|u^\e|<C}+ C\int_0^t \|v^\e-v\|_{L^1(\Om)} + f_\Maxi\int_0^t \|u^\e-u\|_{L^1(\Om)}.
\end{align}
Applying Gronwall's lemma \eqref{eq: discr Gronwall} along with $\|v^\e-v\|_{L^1(Q)}\to 0$ we get that $ \|(u^\e-u)(t)\|_{L^1(\Om)}\to 0$ for all $t>0$, which further implies that $ \|(u^\e-u)(t)\|_{L^2(\Om)}\to 0$ since $0\leq u(t),\, u^\e(t)\leq \Breve{u}<C$. It also implies that $ \|\Phi(u^\e)-\Phi(u)\|_{L^2(\Om)}\to 0$ since $\Phi$ is Lipschitz in $[0,\Breve{u}]$. This proves the result.
\end{proof}

\begin{proof}[\underline{Proof of \eqref{eq:Conv-notH1} in \Cref{theo:time_convergence}}]
We choose $\e>0$ small enough such that along the subsequence in \Cref{lemma:eps-conv} we have 
\begin{subequations}
\begin{align}
         \int_0^T \left[\|u^\e-u\|^2 + \|w^\e-w\|^2+ \|v^\e-v\|^2\right]\leq \tfrac{1}{2}\delta.
\end{align}
for some arbitrary $\d>0$.
For this fixed $\e>0$, noting that $u^\e_0\in H^1_0(\Om)$, one can choose a time-step $\t>0$ small enough such that by \eqref{eq:Conv-H1} one has
\begin{align}
        \sum_{n=0}^{N_\t} \int_{t_{n-1}}^{t_n}\left[\|u_n^\e -u^\e(t)\|^2 + \|w^\e_n-w^\e(t)\|^2+  \|v^\e_n-v^\e(t)\|^2\right]\dd t\leq \tfrac{1}{2}\delta.
\end{align}
Combining these estimates, one finds the desired subsequence $(\e,\t)\to (0,0)$ such that \eqref{eq:Conv-notH1} holds.
\end{subequations}
\end{proof}

\section{Linearisation} \label{sec: linear}

We have shown that the time-discretised system \eqref{eq: time-discrete weak system} is well-posed and that its solutions possess the qualitative behaviour we expect from the time-continuous system. In this section, we propose linearisation schemes and prove their well-posedness and convergence.   
Recall that $\Phi$ is possibly not Lipschitz continuous if $b=1$.
However,  $u_n$ takes values in $[0,\Breve{u}]$ where $\Breve{u}<b$  is a uniform a priori computable upper bound (see \Cref{rem: computable}), and $\Phi$ is Lipschitz in $[0,\Breve{u}]$. Hence, we can regularize $\Phi$ as follows.

\begin{definition}[Regularization of $\Phi$] \label{def: reg Phi}
If $b=1$ and with $\Breve{u}>0$ given in \eqref{eq:uup}, let the function $\Phib\colon\R^+\to \R^+$ be defined as
    \begin{equation}
    \begin{split}
        \Phib(u) = \begin{cases} \Phi(u), &\text{if } u \leq \Breve{u}, \\
        \Phi'(\Breve{u})(u-\Breve{u}) + \Phi(\Breve{u}), &\text{if } u\geq \Breve{u}.\end{cases}
    \end{split}
\end{equation}
If $b=\infty$, we set $\Phib=\Phi$.
\end{definition}

Recalling that $\Phi(u)=w$ posseses space regularity, 
we propose an iterative linearisation scheme to solve  \eqref{eq: time-discrete weak 1} which splits the equation into two coupled equations. The iterations are obtained by solving the following.  
\begin{problem*}[The splitting linearisation]
Let $n\in\N$ and $i \in \N_0$ be fixed, and assume $u_{n-1}, v_{n-1}\in L^2(\Om)$ and $(u^{i-1}_n, w^{i-1}_n) \in  L^2(\Om)\times \H$ be given, satisfying $u_{n-1}, u^{i-1}_n \leq \Breve{u}$. Find the pair $(u^i_n, w^i_n) \in L^2(\Om)\times \H$ such that, for all $\phi \in \H$ and $\xi \in L^2(\Om)$ it holds that 
\begin{subequations} \label{eq: linear weak}
    \begin{align}
        &\brac{\frac{1}{\t} \brac{\tilde{u}^i_n - u_{n-1}}, \phi} +  \brac{\nabla w^i_n, \nabla \phi}= \brac{f(v_{n-1}) \tilde{u}^i_n,\phi} 
        ,\label{eq: linear weak 1} \\
        & \brac{L^i_n (\tilde{u}^{i}_n - u^{i-1}_n), \xi} =  \brac{ w^i_n - \Phib\brac{u^{i-1}_{n}}, \xi}
        , \label{eq: linear weak 2} \\
        &\ u^i_n = \posi{\Tilde{u}^i_n} \text{ a.e. in } \Om, \label{eq: linear weak 3}
    \end{align}
\end{subequations}
 for some specific choice of a bounded function $L^i_n:\Om\to \R^+$, which depends only on iterates up to $u^{i-1}_n$ but not on $u^i_n$. The iteration starts with the initial guess $u^{0}_n=u_{n-1}$.
\end{problem*}

Such a splitting method was first proposed in \cite[Section 4.2]{cances2021error} for the $L$-scheme assuming that $\Phi$ is Lipschitz. Here, we generalize the results.

\begin{remark}[Positivity of $u^i_n$] \label{rem: un positive}
To shorten the proofs, throughout this section, we will simplify
\eqref{eq: linear weak} (where we first determine $\Tilde{u}^i_n$ and then set $u^i_n = \posi{\tilde{u}^i_n}$) by referring to $\tilde{u}^i_n$ interchangeably as $u^i_n$. All  inequalities and results in this section remain valid since
\begin{equation}
   \norm{u^i_n - u_n}_{L^p(\Omega)} =  \norm{\posi{\Tilde{u}^i_n} - u_n}_{L^p(\Omega)} \leq \norm{\Tilde{u}^i_n - u_n}_{L^p(\Omega)},
\end{equation}
for all $p\geq 1$.
Indeed, as
$u_n \geq 0$ which gives $\posi{\Tilde{u}^i_n} - u_n=\Tilde{u}^i_n - u_n$ if $\tilde{u}^i_n\geq 0$ and $\Tilde{u}^i_n - u_n< \posi{\Tilde{u}^i_n} - u_n\leq 0$ if $\tilde{u}^i_n< 0$. 
The reason we introduce the formulation \eqref{eq: linear weak} is that it guarantees that $u^i_n \geq 0$, which is important for the numerical implementation.
\end{remark}

\noindent Before proving results for particular linearisation schemes, we show that the regularization $\Phib$ does not alter the solution $u_n$. This is obvious for $b=\infty$, while for $b=1$ it follows from the proposition below. 

\begin{proposition}[Consistency of the regularized $\Phib$]
Let $b=1$, and $\Phib$ the regularized approximation of $\Phi$ given in \Cref{def: reg Phi}. Then, the solution of \eqref{eq: linear weak} coincides with the solution to \eqref{eq: time-discrete weak system}.
\end{proposition}

\begin{proof}
Suppose $u_n$ and $\Tilde{u}_n$ are the weak solution of
\begin{align}
        \frac{1}{\t}(\Tilde{u}_n-u_{n-1})&=\D  \Phib(\Tilde{u}_n) + f(v_{n-1}) \Tilde{u}_n,\label{eq: uregL}\\
        \frac{1}{\t}(u_n-u_{n-1})&=\D  \Phi(u_n) + f(v_{n-1}) u_n.\label{eq: uregno}
\end{align}
\noindent Since $u_n \leq \Breve{u}$ by \Cref{thm: main time 1}, one has $\Phib(u_n) = \Phi(u_n)$, i.e. 
$u_n$ is a solution of Equation \eqref{eq: uregL}. However, this solution is unique due to 
\Cref{thm: main time 1} which implies that $\Tilde{u}_n = u_n$.
\end{proof}

\noindent To show that the linearisation scheme is well-defined, we  prove that if 
it converges, the limit is indeed a solution of the time-discretised equation \eqref{eq: time-discrete weak 1}.

\begin{proposition}[Consistency of the linearisation scheme] \label{prop: lin conv sol}
    Let $u^i_n$ be uniformly bounded with respect to $i\in \N$ in $L^2(\Omega)$, $u^i_n \to \Tilde{u}_n$ strongly in $L^1(\Omega)$, and $w^i_n \to \Tilde{w}_n$ strongly in $H^1_0(\Omega)$. Then $\Tilde{u}_n$ is the weak solution to the  time-discretised equation \eqref{eq: time-discrete weak 1} and $\Tilde{w}_n=\Phi(\Tilde{u}_n)$ a.e. in $\Om$.
\end{proposition}

\Cref{thm: conv L-scheme,thm: conv M-scheme} will show that the hypotheses of \Cref{prop: lin conv sol} are indeed satisfied for the L- and M-schemes. Hence, the iterates $(u^i_n,w^i_n)$ converge to the time-discrete solutions.

\begin{proof}
First we observe that $(u_n^i,\phi)\to (\Tilde{u}_n,\phi)$ for all $\phi\in L^2(\Om)$, since $u^i_n$ is bounded in $L^2(\Om)$ and  $u^i_n \to \Tilde{u}_n$ strongly in $L^1(\Om)$. Hence,  taking the limit in \eqref{eq: linear weak 1} implies that 
 \begin{equation} \label{eq: linear weak 1 limit}
     \brac{h_{n-1}\Tilde{u}_n - u_{n-1}, \phi} + \t \brac{\nabla \Tilde{w}_n, \nabla \phi} = 0\qquad \forall \phi \in \H,
 \end{equation}
where $h_{n-1}$ is defined in \eqref{eq:def_hn}. Similarly, taking the limit $i \to \infty$ in \eqref{eq: linear weak 2} we get
 \begin{equation}
     \brac{\Tilde{w}_n - \Phib(\Tilde{u_n}), \xi} = 0 \qquad \forall \xi \in L^2(\Om).
 \end{equation}
Here, we used that  $\Phib$ is Lipschitz continuous implying that  $(\Phib(u^i_n),\xi)\to (\Phib(\tilde{u}_n),\xi)$, and that $L^i_n$ is bounded which yields 
$\brac{L^i_n \brac{u^i_n - u^{i-1}_n}, \xi} \to 0$ for all $\xi \in L^2(\Om)$. We conclude that $\Tilde{w}_n = \Phib(\Tilde{u}_n)$ a.e., which allows us to substitute it back into Equation \eqref{eq: linear weak 1 limit} and hence, 
\begin{equation}
    \brac{h_{n-1}\Tilde{u}_n - u_{n-1}, \phi} + \t \brac{\nabla \Phib\brac{\Tilde{u}_n}, \nabla \phi} = 0\qquad \forall \phi \in \H.
\end{equation}
This coincides with the time-discretised equation for $u_n$, and thus $\Tilde{u}_n = u_n$, as solutions are unique.
\end{proof}

We can identify the linearisation schemes mentioned in \Cref{sec: 1}
as special cases of \eqref{eq: linear weak}:
\begin{subequations}\label{eq:different_schemes}
\begin{align}
    &\textbf{Newton scheme}: L^i_n:= \Phib'(u^{i-1}_n),\\
    &\textbf{L-scheme}: \qquad \quad L^i_n:= L, \quad \text{for a constant } L>0,\label{eq:different_schemesL}\\
    &\textbf{M-scheme}: \qquad \;\;L^i_n:= \max\{\Phib'(u^{i-1}_n)+ M\t^\g ,\, 2M \t^\g \},  \text{ for constants } M>0,\; \g\in (0,1].\label{eq:different_schemesM}
\end{align}
\end{subequations}
In the sequel, we consider the L- and M-schemes as our main focus will be on degenerate problems.
We denote the errors of the iterates at the $n^{\text{th}}$ time step by 
\begin{align}\label{eq:def_errors}
    e^i_u = u^i_n - u_n, \quad e^i_w = w^i_n - w_n,
\end{align}
where $u_n, w_n, u_n^i, w_n^i$ are the solutions of \eqref{eq: time-discrete} and \eqref{eq: linear} respectively. 

The following theorems provide the main convergence results for the L- and M-scheme, their proofs are given in Subsections \ref{sec: Lscheme} and \ref{sec: Mscheme}.
The results for both schemes are similar, but the proofs for the L-scheme are more straightforward.

\begin{theorem}[Convergence of the L-scheme] \label{thm: conv L-scheme}
For $\t< 1\slash f_\Maxi$ there exist unique solutions $\{(u^i_n, w^i_n)\}_{i\in\N}\subset L^2(\Om)\times \H$ of \eqref{eq: linear weak} with \eqref{eq:different_schemesL}, i.e. $L^i_n:= L$. Furthermore, if $L>\sup \Phib'$, then 
$u^i_n \to u_n$ in $L^1(\Om)$ and $w^i_n \to w_n$ in $\H$.
\\
For non-degenerate problems, i.e. if $\inf\Phi'=\phi_\Mini>0$, the error-norm is a strict contraction $$\|(e^i_u,e^i_w)\|_L\leq \a\|(e^{i-1}_u,e^{i-1}_w)\|_L$$ with   rate
$\alpha = \sqrt{\frac{L}{L+ \phi_{\Mini}}},$
where
\begin{equation}\label{eq:norm_Lscheme}
    \norm{\brac{e^i_u, e^i_w}}_{L}^2 \coloneqq \int_\Om h_{n-1}|e^i_u|^2 + \frac{2\t}{L + \phi_\Mini} \norm{\nabla e^i_w}^2_{L^2(\Om)}.
\end{equation}
\end{theorem}

For the M-scheme, we need to impose an additional assumption.

\begin{enumerate}[label=(A\theAssEx)]
 \setlength\itemsep{-0.2em}
 \item For a given $n\in \N$, there exists $\Lambda \geq 0$ and $\g\in (0,1]$ such that $\|u_{n} - u_{n-1}\|_{L^{\infty}(\Om)} \leq \Lambda \t^\g $.\label{ass: extra} \stepcounter{AssEx}
\end{enumerate}

\begin{remark}[Assumption \ref{ass: extra}]
 Assumption \ref{ass: extra}   was used in \cite{MITRA20191722} with $\g=1$ in the context of nonlinear diffusion problems, and this property was proven for a particular case in Proposition 3.1, but not for porous medium type diffusion. Note that \ref{ass: extra} with $\g=1$ is the time-discrete counterpart of the regularity assumption $\p_t u\in L^\infty(\Om)$. But for degenerate problems this is typically not satisfied. However, solutions of porous medium type equations are H\"older continuous, and for degenerate and singular systems of the form \eqref{eq: main}, the H\"older continuity of solutions was shown in \cite{HISSINKMULLER2}. Hence, Assumption \ref{ass: extra} is expected to hold as a time-discrete counterpart of the H\"older continuity with exponent $\g\in (0,1]$.
\end{remark}

\begin{theorem}[Convergence of the M-scheme]\label{thm: conv M-scheme}
For $\t< 1\slash f_\Maxi$ there exist unique solutions $\{(u^i_n, w^i_n)\}_{i\in\N}\subset L^2(\Om)\times \H$ of \eqref{eq: linear weak} \eqref{eq: linear weak} with \eqref{eq:different_schemesM}. Furthermore, assume that \ref{ass: extra} holds, take $M>M_{0}\coloneqq \|\Phib^{\prime}\|_{\rm Lip} \Lambda$,  and let $\{u_n^i\}_{i\in\N}$ satisfy 
\begin{align}\label{eq: M-scheme assumption}
   \|u^i_n-u_n\|_{L^\infty(\Om)} \leq \Lambda \t^\g  \text{ for all } i\in \N.
\end{align}
Then, 
$u^i_n \to u_n$ in $L^1(\Om)$ and $w^i_n \to w_n$ in $\H$. \\
For non-degenerate problems, i.e. if $\inf\Phi'=\phi_\Mini>0$, and if $\t < \left(\phi_\Mini/M\right)^{\frac{1}{\g}}$, then 
the error-norm is a strict contraction, 
$$\|(e^i_u,e^i_w)\|_M\leq \a\|(e^{i-1}_u,e^{i-1}_w)\|_M,$$ 
with rate
$\a= \frac{2M\t^\g }{\phi_\Mini+M\t^\g },$
where
\begin{equation}\label{eq:cond_convM}
    \norm{\brac{e^i_u, e^i_w}}_{M}^2 \coloneqq \int_\Om h_{n-1}|e^i_u|^2 + \frac{2\t}{\phi_{\Mini} + M\t^\g } \norm{\nabla e^i_w}^2_{L^2(\Om)}.
\end{equation}
\end{theorem}

For the convergence of $(u^i_n,w^i_n)$ in  the $L^\infty$-norm, see \Cref{lem: Linf conv Mscheme}.

\begin{remark}[Boundedness condition \eqref{eq: M-scheme assumption} and contraction in $L^\infty$]
   In the non-degenerate case, i.e. if $\phi_\Mini>0$,  the boundedness condition \eqref{eq: M-scheme assumption} follows from \ref{ass: extra} for time-step sizes $\t\leq (\phi_\Mini/3M)^{\frac{1}{\g}}$, as stated in \Cref{lem: Linf conv Mscheme}. In fact,  \Cref{lem: Linf conv Mscheme} even provides linear convergence of $u^i_n$ to $u_n$ in $L^\infty(\Om)$ with a contraction rate that scales with $\t^\g $. For the case of singular diffusion, a proof of  \eqref{eq: M-scheme assumption} was given in \cite[Lemma 3.1]{MITRA20191722}. 
   We expect that the result also holds in our case. However. since it is not the main focus of this work, we state it as an assumption.
\end{remark}

\begin{remark}[Comparison L- and M-scheme]
   Note that the extra assumptions \ref{ass: extra} and \eqref{eq: M-scheme assumption} are not required for the L-scheme, and hence, the L-scheme is expected to be more robust than the M-scheme. However, this comes at the cost of being considerably slower than the Newton scheme.  On the other hand, the assumptions required for M-scheme are expected to hold for problems such as \eqref{eq: main}. Since the contraction rate for the M-scheme scales with $\t $, for practical purposes the M-scheme results in a more competitive iterative solver than the L-scheme.
\end{remark}

We first prove the existence and uniqueness results stated in \Cref{thm: conv L-scheme,thm: conv M-scheme}. Recall that $L^i_n=L$ is constant for the L-scheme and $\sup\Phib'$ is bounded due to assumption \ref{ass: 1} and the construction of $\Phib$ in \Cref{def: reg Phi}. The proof of the following lemma applies to both schemes.
\begin{lemma}[Existence-uniqueness] \label{lem: exist L}
For $\t< 1\slash f_M$, the system of equations \eqref{eq: linear weak} with 
$$L^i_n \coloneqq L > \sup\Phib'\qquad \text{or}\qquad 
L^i_n\coloneqq \max\{\Phib'(u^{i-1}_n)+ M\t^\g ,\, 2M \t^\g \}
$$ has a unique solution.
\end{lemma}

\begin{proof}
We eliminate $\Tilde{u}^i_n$ in  \eqref{eq: linear weak 1} through \eqref{eq: linear weak 2} and find
\begin{equation}
    \left(\frac{h_{n-1}}{L^i_n}w_{n}^{i}, \phi\right) + \t (\nabla w_n^i, \nabla \phi) = (g_n^i, \phi) \quad \text{ for all $\phi \in H_{0}^{1}(\Om)$} \label{eq: combined weak form ML},
\end{equation}
where $g_n^i = \frac{h_{n-1}}{L^i_n}\Phib(u^{i-1}_n) - (h_{n-1}u^{i-1}_n - u_{n-1})$.
Consider the bilinear form $B(w,\phi) = ((h_{n-1}/L^i_n)\, w,\phi) + \t (\nabla w, \nabla \phi)$ and the linear functional $l(\phi) = (g_n^i,\phi)$. We observe that $L^i_n$ is constant, or bounded from above and below by positive constants in case of the M-scheme, see \eqref{eq:different_schemes}, and $0<h_{n-1}<1$ due to $\t < 1/f_\Maxi$. Hence, using the Cauchy-Schwarz and Poincar\'e inequality implies that $B$ is coercive and bounded and $l$ is a bounded linear functional on $H_0^1(\Om)$. The Lax-Milgram theorem now provides the existence and uniqueness of a solution $w^i_n\in \H$. The existence and uniqueness of $\Tilde{u}_n^i\in L^2(\Om)$ then follows from 
\eqref{eq: linear weak 2}, while $u^i_n$ can be found through \eqref{eq: linear weak 3}.
\end{proof}

\subsection{L-scheme}\label{sec: Lscheme}
 First, we show that the solutions of the L-scheme converge to the time-discrete solutions $u_n$ and $w_n$.

\begin{lemma}[Convergence of the L-scheme] \label{lem: conv L-scheme}
Under the assumptions of \Cref{thm: conv L-scheme}, the stated convergence results hold.

\end{lemma}

\begin{proof}
We use ideas from the proof of Lemma 2.6 in \cite{cances2021error}. Subtracting  \eqref{eq: time-discrete weak 1} from Equation \eqref{eq: linear weak 1} we find
\begin{equation}
    \brac{h_{n-1}e^i_u, \phi} + \t \brac{\nabla e^i_w, \nabla \phi} = 0,
\end{equation}
where $e^i_u = u^i_n - u_n$ and $e^i_w = w^i_n - w_n$, see \eqref{eq:def_errors}. By adding and subtracting $Lu_n$, and adding and subtracting $w_n=\Phib(u_n)$ on the right hand side of \eqref{eq: linear weak 2}, we can rewrite it as
\begin{equation}
    \brac{L \brac{e^i_u - e^{i-1}_u}, \xi} = \brac{e^i_w - \d \Phib^{i-1}, \xi}
\end{equation}
where $\d \Phib^{i-1} \coloneqq \Phib \brac{u^{i-1}_n} - \Phib(u_n)$. Choosing $\phi = e^i_w \in \H$ and $\xi = h_{n-1}e^i_u \in L^2(\Om)$, we combine the two equations and obtain
\begin{equation}
    \brac{h_{n-1}L \brac{e^i_u - e^{i-1}_u}, e^i_u} + \brac{h_{n-1}\d \Phib^{i-1}, e^i_u} + \t \norm{\nabla e^i_w}_{L^2(\Om)}^2 = 0.
\end{equation}
This is rewritten as
\begin{equation}\label{eq:proof_convL}
    \brac{ h_{n-1}\brac{L - \brac{\frac{\d \Phib}{\d u}}^{i-1}} \brac{e^{i}_u - e^{i-1}_u}, e^i_u } + \brac{h_{n-1}\brac{\frac{\d \Phib}{\d u}}^{i-1} e^{i}_u, e^i_u} + \t \norm{\nabla e^i_w}_{L^2(\Om)}^2 = 0,
\end{equation}
where
\begin{equation}\label{eq: def dphi}
    \brac{\frac{\d \Phib}{\d u}}^{i-1} \coloneqq \frac{\Phib \brac{u^{i-1}_n} - \Phib(u_n)}{u^{i-1}_n - u_n}=\frac{\d \Phib^{i-1}}{e^{i-1}_u}.
\end{equation}
Using the identity $(a-b)a = \frac{1}{2}\brac{a^2 - b^2 + \brac{a-b}^2}$ with $a = e^i_u$ and $b = e^{i-1}_u$ we rewrite the first term in \eqref{eq:proof_convL} and obtain
\begin{equation} \label{eq: L-scheme equality}
\begin{split}
    &\frac{1}{2} \int_{\Om} h_{n-1}\brac{L + \brac{\frac{\d \Phib}{\d u}}^{i-1}} |e^i_u|^2 + \frac{1}{2} \int_{\Om} h_{n-1}\brac{L - \brac{\frac{\d \Phib}{\d u}}^{i-1}} |e^i_u - e^{i-1}_u|^2 \\
    &+ \t \norm{\nabla e^i_w}_{L^2(\Om)}^2 = \frac{1}{2} \int_{\Om} h_{n-1}\brac{L - \brac{\frac{\d \Phib}{\d u}}^{i-1}} |e^{i-1}_u|^2.
\end{split}
\end{equation}
Note that $L > \sup\Phib'$ and assumption \ref{ass: 1} imply that
\begin{equation}
    0 \leq \phi_\Mini \leq \brac{\frac{\d \Phib}{\d u}}^{i-1} < L.
\end{equation}
Combining this with equation \eqref{eq: L-scheme equality} we find that
\begin{equation} \label{eq: L-scheme inequality}
    \frac{L+\phi_\Mini}{2} \int_\Om h_{n-1}|e^i_u|^2 + \frac{\varepsilon}{2}\int_{\Om}h_{n-1}|e^i_u - e^{i-1}_u|^2 + \t \norm{\nabla e^i_w}^2_{L^2(\Om)} \leq \frac{L}{2} \int_{\Om}h_{n-1}|e^{i-1}_u|^2,
\end{equation}
where  
\begin{equation}\label{def: epsilon}
    \varepsilon \coloneqq   L-\sup \Phib' \leq  \brac{L - \brac{\frac{\d \Phib}{\d u}}^{i-1}} .
\end{equation}
Note that the norm $\sqrt{\int_\Om h_{n-1}|e^i_u|^2}$ is equivalent to $\norm{e^i_u}_{L^2(\Om)}$, since $0<h_{n-1} <1$ for $\t < 1 \slash \norm{f}_{L^\infty}$. 

In the non-degenerate case, i.e. if  $\phi_\Mini > 0$, we obtain a contraction as the second term in \eqref{eq: L-scheme inequality} is positive, $\norm{\brac{e^i_u, e^i_w}}_{L} \leq \sqrt{\frac{L}{L + \phi_\Mini}} \norm{\brac{e^{i-1}_u, e^{i-1}_w}}_{L},$ with the norm  defined in \eqref{eq:norm_Lscheme}.
Consequently, $u^i_n \to u_n$ in $L^2(\Om)$ and $w^i_n \to w_n$ in $\H$ by  Banach's fixed-point theorem.

In the degenerate case, we can sum up both sides of Equation \eqref{eq: L-scheme inequality} to find
\begin{equation}
\begin{split}
        0 &\leq \frac{\varepsilon}{2}\sum_{i=1}^N \brac{\int_{\Om}h_{n-1}|e^i_u - e^{i-1}_u|^2} + \t \sum_{i=1}^N \norm{\nabla e^i_w}_{L^2(\Om)}^2 \\&\leq \frac{L}{2} \int_{\Om}h_{n-1}|e^0_u|^2 - \frac{L}{2} \int_{\Om}h_{n-1}|e^N_u|^2 < \infty.
\end{split}
\end{equation}
\noindent Hence, taking the limit $N \to \infty$ we conclude that of both sums must go to 0, yielding
\begin{subequations}
\begin{align}
    &\norm{\nabla e^i_w}_{L^{2}(\Om)} \to 0, \label{eq: L-scheme w conv}\\
    &\norm{e^{i}_u - e^{i-1}_u}_{L^2(\Om)} \label{eq: L-scheme eiu diff} \to 0.
\end{align}
\end{subequations}
\noindent From \eqref{eq: L-scheme w conv} it follows that  $w^i_n \to w_n$ in $\H$. Moreover, we  can rewrite \eqref{eq: L-scheme eiu diff} and use the strong form of \eqref{eq: linear weak 2} to find
\begin{equation}
    \norm{\frac{1}{L}\brac{w^i_n - \Phib(u^{i-1}_n)}}_{L^2(\Om)} = \norm{u^i_n - u^{i-1}_n}_{L^2(\Om)} = \norm{e^i_u - e^{i-1}_u}_{L^2(\Om)} \to 0.
\end{equation}
Hence, $\Phib(u^i_n) \to w_n = \Phib(u_n)$ in $L^2(\Om)$, as $L > 0$. Finally, to prove that $u^i_n \to u_n$ in $L^1(\Om)$ we use \Cref{lemma:imprtnt_conv} replacing the function $\psi$ with $\Phib$.

The result for the positive part of $u^i_n$ follows from \Cref{rem: un positive}, which completes the proof.
\end{proof}

\begin{proof}[Proof of \Cref{thm: conv L-scheme}]
    The existence and uniqueness of solutions $\{(u^i_n, w^i_n)\}_{i\in\N}\subset L^2(\Om)\times \H$  follows from \Cref{lem: exist L} and the convergence  from \Cref{lem: conv L-scheme}.
\end{proof} 

We have shown that the L-scheme converges and that the error is a strict contraction in the non-degenerate case. Unfortunately, the contraction rate is very close to 1 if $\phi_\Mini>0$ is small compared to $L$, as reported in \cite{PS2,MITRA20191722,Stokke2023}. 
A closer inspection indicates that setting $L^i_n>\sup \Phib'$ everywhere in the domain is superfluous and that this is the main reason for the slow convergence rate. 
To overcome this drawback 
we aim to modify the L-scheme such that it is stable but converges fast. This leads us to the M-scheme, first introduced in \cite{MITRA20191722}.

\subsection{M-scheme}\label{sec: Mscheme}
Note that in contrast to the L-scheme, now $L^i_n$ is a function of the previous iterate $u^{i-1}_n$.
We first derive two useful estimates that are needed to prove the main convergence results.
\begin{lemma}[Some useful inequalities]\label{lemma: inequalities}
Let \ref{ass: 1} and \eqref{eq: M-scheme assumption} hold and $M \geq M_{0}=\|\Phib^{\prime}\|_{\rm Lip}\Lambda$. 
Then, the following inequalities hold:
\begin{subequations}
\begin{align}
    &L^i_n \geq 2M\t^\g  \label{ineq: 1}, \\
&0 \leq (M-M_0)\t^\g \leq L^i_n - \brac{\frac{\d \Phib}{\d u}}^{i-1} \leq 2M\t^\g  \label{ineq: 2},
\end{align}
\end{subequations}
where $\brac{\frac{\d \Phib}{\d u}}^{i-1}$ was defined in \eqref{eq: def dphi}.
\end{lemma}
\begin{proof}
Note that \eqref{ineq: 1} is an immediate consequence of the definition of $L^i_n$. 
To prove \eqref{ineq: 2}, we first note that 
\begin{equation}\label{eq: ineq dphib}
    \brac{\frac{\d \Phib}{\d u}}^{i-1} = \Phib^{\prime}(\zeta),
\end{equation}
for some $\zeta \in I(u^{i-1}_n, u_n)$ by the mean-value theorem. Moreover, for any $\zeta \in I(u^{i-1}_n, u_n)$, we have
\begin{equation} \label{ineq: proof}
        |\Phib^{\prime}(u^{i-1}_n) - \Phib^{\prime}(\zeta)| \leq \|\Phib^{\prime}\|_{\rm Lip}\|u^{i-1}_n - \zeta\|_{L^{\infty}(\Om)}\leq \|\Phib^{\prime}\|_{\rm Lip} \Lambda \t  = M_{0}\t^\g ,
\end{equation}
where we used \eqref{eq: M-scheme assumption} in the last inequality. 
If $L^i_n = \Phib^{\prime}(u^{i-1}_n) + M\t^\g $, then $L^i_n - \Phib^{\prime}(\zeta) \geq (M-M_0)\t^\g  \geq 0$. On the other hand, if $L^i_n = 2M\t^\g $, then $\Phib^{\prime}(u^{i-1}_n) \leq M\t^\g $ by the definition of $L^i_n$. Together with \eqref{ineq: proof} we conclude that $\Phib^{\prime}(\zeta) \leq \Phib^{\prime}(u^{i-1}_n) + M_0 \t^\g  \leq (M+M_0)\t^\g $, and thus again $L^i_n - \Phib^{\prime}(\zeta) \geq (M-M_0)\t^\g \geq 0$. 

To derive the upper bound we argue analogously. If $L^i_n = \Phib^{\prime}(u^{i-1}_n) + M\t^\g $, we find
\begin{equation}
    L^i_n - \Phib^{\prime}(\zeta) \leq \|\Phib^{\prime}\|_{\rm Lip}(u^{i-1}_n - \zeta) + M\t^\g  \leq  \|\Phib^{\prime}\|_{\rm Lip} \Lambda\t^\g  + M\t^\g  \leq 2M\t^\g .
\end{equation}
If $L^i_n = 2M\t^\g $, then we have $L^i_n - \Phib^{\prime}(\zeta) \leq L^i_n - \phi_\Mini \leq 2M\t^\g $. Hence, combining all estimates we find
\begin{equation}
    0 \leq (M-M_0)\t^\g  \leq L^i_n - \Phib^{\prime}(\zeta) \leq 2M\t^\g.
\end{equation}
As the estimates hold for any $\zeta \in I(u^{i-1}_n, u_n)$, the statement follows from \eqref{eq: ineq dphib}.

\end{proof}

Next, we prove the first convergence result in $L^\infty$ for non-degenerate problems. 

\begin{proposition}[$L^{\infty}$ convergence of $u^i_n$] \label{lem: Linf conv Mscheme}
Assume that $\inf \Phi'=\phi_{\Mini} > 0$ and \ref{ass: extra} holds. Then, for $M>M_{0}\coloneqq \|\Phib^{\prime}\|_{\rm Lip} \Lambda$, one has
\begin{equation}
    \|w^i_n-w_n\|_{L^\infty(\Om)}\leq 2M \t^\g  \|u^{i-1}_n-u_n\|_{L^\infty(\Om)}.\label{ineq:  theorem 1}
\end{equation}
 Moreover, if $\t < (\phi_\Mini/M)^{\frac{1}{\g}}$, then
\begin{equation}
\|u^i_n-u_n\|_{L^\infty(\Om)}\leq \frac{4M \t^\g }{\phi_\Mini+M\t^\g } \|u^{i-1}_n-u_n\|_{L^\infty(\Om)}. \label{ineq:  theorem 2}
\end{equation}  
Therefore, if $\t < (\phi_\Mini/(3M))^{\frac{1}{\g}}$, then $u^i_n$ converges linearly in $L^\infty(\Om)$ to $u_n$ and  the uniform boundedness of the iterates $\|u^{i}_n-u_n\|_{L^\infty(\Om)} \leq \Lambda \t^\g $ in \eqref{eq: M-scheme assumption} holds.
\end{proposition}

\begin{proof}
We prove the statement by induction in $i\in \N$. For $i = 1$ it is satisfied by assumption \ref{ass: extra} and \Cref{rem: un positive}. \Cref{rem: un positive} will also be used in the following estimates, i.e. taking the positive part of $u_n$ does not alter the inequalities.
For the induction step we assume $\|u^{i-1}_n-u_n\|_{L^\infty(\Om)} \leq \Lambda \t^\g $, which allows us to use \Cref{lemma: inequalities}. We split the proof into two parts. First, we show that  \eqref{ineq: theorem 1} holds and subsequently, we deduce from it \eqref{ineq:  theorem 2}. 

\textbf{(Step 1:)} Note that \eqref{ineq:  theorem 1} is equivalent to showing that $[e^{i}_w -a]_+ = 0$ and $[e^{i}_w + a]_- = 0$ for a specific $a>0$. 
We subtract \eqref{eq: time-discrete weak 1} from Equation \eqref{eq: linear weak 1} and rewrite \eqref{eq: linear weak 2} as in the proof of \Cref{lem: conv L-scheme}, which yields
\begin{subequations} \label{eq: weak diff}
    \begin{align}
        (h_{n-1}e^i_u, \phi) + \t(\nabla e^i_w, \nabla \phi) &= 0, \label{eq: weak diff 1}\\
        \brac{L^i_n \brac{e^i_u - e^{i-1}_u}, \xi} &= \brac{e^i_w - \d \Phib^{i-1}, \xi} \label{eq: weak diff 2}.
    \end{align}
\end{subequations}
 Choosing $\phi = [e^{i}_{w} -a]_{+} \in \H$ in \eqref{eq: weak diff 1} yields
\begin{equation}
    (h_{n-1}e^i_u, [e^{i}_{w} -a]_{+}) + \t(\nabla e^i_w, \nabla [e^{i}_{w} -a]_{+}) = 0,
\end{equation}
and since the second term is positive, we find that
\begin{equation} \label{eq: linf proof leq 0}
    (h_{n-1}e^i_u, [e^{i}_{w} -a]_{+}) \leq 0.
\end{equation}
To eliminate $e^i_u$ we observe that equation \eqref{eq: weak diff 2} implies that
\begin{equation}\label{eq: pf bound}
    e^i_u = \frac{L^i_n - \brac{\frac{\d \Phib}{\d u}}^{i-1}}{L^i_n}e^{i-1}_{u} + \frac{e^i_w}{L^i_n},
\end{equation}
almost everywhere. Combining \eqref{eq: linf proof leq 0} and \eqref{eq: pf bound} yields
\begin{equation}
    \int_{\Om}\frac{h_{n-1}}{L^{i}_n}[e^i_w-a][e^i_w-a]_{+} + \int_{\Om}\frac{h_{n-1}}{L^i_n}\left(a+\brac{L^i_n - \brac{\frac{\d \Phib}{\d u}}^{i-1}}e^{i-1}_{u} \right)[e^i_w - a]_{+} \leq 0.
\end{equation}
The first term is positive and the second term can be made positive by choosing
\begin{equation}
a = 2M\t^\g  \infnorm{u^{i-1}_n - u_n},
\end{equation}
as $0\leq L^i_n - \brac{\frac{\d \Phib}{\d u}}^{i-1} \leq 2M\t^\g$ by \Cref{lemma: inequalities}. Hence, we find that $[e^{i}_w - a]_{+} = 0$. The proof for $\phi = [e^i_w + a]_{-} \in \H$ is analogous which proves \eqref{ineq:  theorem 1}.

\textbf{(Step 2:)} To show \eqref{ineq:  theorem 2} we 
again note that $L^i_n - \brac{\frac{\d \Phib}{\d u}}^{i-1} \leq 2M\t^\g $ by \Cref{lemma: inequalities} and that in the non-degenerate case, we have
\begin{equation}
    \frac{1}{L^i_n} \leq \min\left\{\frac{1}{2M\t^\g }, \frac{1}{\phi_{\Mini}+M\t^\g }\right\}.
\end{equation}
Hence, using \eqref{ineq:  theorem 1} in 
Equation \eqref{eq: pf bound} implies that 
\begin{equation}
    \begin{split}
        \|e^i_u\|_{L^\infty(\Om)} &\leq \min\left\{\frac{1}{2M\t^\g }, \frac{1}{\phi_\Mini+M\t^\g }\right\} (2M\t^\g  \|e^{i-1}_{u}\|_{L^\infty(\Om)} + 2M\t^\g  \|e^{i-1}_{u}\|_{L^\infty(\Om)}) \\
        &= \min\left\{2, \frac{4M\t^\g }{\phi_\Mini+M\t^\g }\right\}\|e^{i-1}_{u}\|_{L^\infty(\Om)}.
    \end{split}
\end{equation}
Consequently, if  $\t  < (\frac{\phi_\Mini}{M})^{\frac{1}{\g}}$,  we get \eqref{ineq:  theorem 2}, and

Finally, note that the linear convergence and uniform $L^\infty$-bound of the iterates \eqref{eq: M-scheme assumption} follows if $\t  < (\frac{\phi_\Mini}{3M})^{\frac{1}{\g}}$. Indeed, with the contraction rate $\bar{\a}={4M\t^\g }/({\phi_\Mini+M\t^\g })<1$ one has
\[\|e^i_u\|_{L^\infty(\Om)}< \bar{\a}\|e^{i-1}_u\|_{L^\infty(\Om)}\leq \dots \leq \bar{\a}^i\|e^0_u\|_{L^\infty(\Om)}=\bar{\a}^i\|u_{n-1}-u_n\|_{L^\infty(\Om)}\leq \Lambda \t^\g ,\]
the last inequality resulting from \ref{ass: extra}.
\end{proof}

Finally, we prove the convergence result for the M-scheme similar to \Cref{lem: conv L-scheme} for the L-scheme. The proof is analogous, but we obtain a better contraction rate in the non-degenerate case as the time step $\t$ is made smaller.

\begin{lemma}[Convergence of the M-scheme]\label{lem: conv M-scheme}
Under the assumptions of \Cref{thm: conv M-scheme} with $M>M_0:=\|\Phib'\|_{\rm Lip}\Lambda$, the   stated convergence results hold.
\end{lemma}

\begin{proof}
As in the proof of \Cref{lem: Linf conv Mscheme}, we subtract \eqref{eq: time-discrete weak 1} from Equation \eqref{eq: linear weak 1} and rewrite \eqref{eq: linear weak 2} which yields
\begin{subequations} \label{eq: weak diff sys2}
    \begin{align}
        &(h_{n-1}e^i_u, \phi) + \t(\nabla e^i_w, \nabla \phi) = 0, \label{eq: weak diff 12}\\
        &(L^i_n e^i_u, \xi) = (e^i_w, \xi) + ((L^i_n - \Phib^{\prime}(\zeta))e^{i-1}_{u}, \xi) \label{eq: weak diff 22},
    \end{align}
\end{subequations}
where $\zeta \in I(u^{i-1}_n, u_{n})$.
Choosing $\phi = e^{i}_{w} \in H_{0}^{1}(\Omega)$ and $\xi = h_{n-1}e_{u}^{i} \in L^{2}(\Omega)$ we  combine the equations and obtain
\begin{equation}\label{eq:proof_convM}
    (h_{n-1}L_{n}^{i}e^{i}_{u},e^{i}_{u}) + \t(\nabla e_{w}^{i}, \nabla e^{i}_{w}) = (h_{n-1}(L_{n}^{i}-\Phib^{\prime}(\zeta))e^{i-1}_{u},e^{i}_{u}).
\end{equation}
 We estimate the right hand side using Young's inequality \eqref{eq: Young} and $0 \leq L^i_n - \Phi^{\prime}(\zeta) \leq 2M\t^\g $, as proven in \Cref{lemma: inequalities},  to find that for any $\rho > 0$,
\begin{equation} \label{ineq: right}
\begin{split}
    (h_{n-1}(L_{n}^{i}-\Phib^{\prime}(\zeta))e^{i-1}_{u},e^{i}_{u}) &\leq 2M\t^\g  \int_{\Om}\sqrt{h_{n-1}}e^{i-1}_{u} \sqrt{h_{n-1}}e^{i}_{u}, \\
    &\leq \frac{M\t^\g }{\rho}\int_{\Om} h_{n-1}|e^{i-1}_{u}|^2 + \rho M \t^\g  \int_{\Om} h_{n-1}|e^{i}_{u}|^2.
\end{split}
\end{equation}
In the non-degenerate case, i.e. $\phi_m>0$,  we estimate the left hand of \eqref{eq:proof_convM} similarly  using that $L^i_n \geq \phi_\Mini + M\t^\g $ and obtain
\begin{equation} \label{ineq: left}
    (\phi_\Mini+M\t^\g ) \int_{\Om}h_{n-1}|e^i_u|^2 + \t \|\nabla e^i_w\|_{L^2(\Om)}^2 \leq (h_{n-1} L_{n}^{i}e^{i}_{u},e^{i}_{u}) + \t \|\nabla e^i_w\|_{L^2(\Om)}^2.
\end{equation}
Combining \eqref{ineq: right} and \eqref{ineq: left} it follows that 
\begin{equation} \label{ineq: comb 1}
    (\phi_\Mini+(1-\rho)M\t^\g ) \int_{\Om}h_{n-1}(e^i_u)^2 + \t \|\nabla e^i_w\|_{L^2(\Om)}^2 \leq \frac{M\t^\g }{\rho}\int_{\Om} h_{n-1}(e^{i-1}_{u})^2.
\end{equation}
 which implies that 
\begin{equation} \label{eq: M-rho norm contraction}
    \norm{(e^i_u, e^i_w)}_{M,\r} \leq \sqrt{\frac{M\t^\g }{\rho(\phi_\Mini+(1-\rho)M\t^\g )}} \norm{(e^{i-1}_u, e^{i-1}_w)}_{M,\r},
\end{equation}
where 
\begin{equation}\label{eq: M-rho norm}
    \norm{(e^i_u, e^i_w)}_{M,\r}^2 \coloneqq  \int_{\Om}h_{n-1}|e^i_u|^2 + \frac{\t}{\phi_\Mini+(1-\rho)M\t^\g } \|\nabla e^i_w\|_{L^2(\Om)}^2.
\end{equation}
 Equation \eqref{eq: M-rho norm} defines a norm if $0 < \rho < 1 + \frac{\phi_\Mini}{M\t^\g }$, and  
choosing $0<\rho=\rho^{\ast} = \frac{1}{2}(1 + \frac{\phi_\Mini}{M\t^\g }) <1 + \frac{\phi_\Mini}{M\t^\g }$, minimizes the contraction rate. Hence, 
\begin{equation}\label{eq:M-contact-best}
   \norm{(e^i_u, e^i_w)}_{M, \rho^\ast} \eqqcolon\norm{(e^i_u, e^i_w)}_{M} \leq \frac{2M\t^\g }{\phi_\Mini + M\t^\g } \norm{(e^{i-1}_u, e^{i-1}_w)}_{M},
\end{equation}
which is a contraction if $\t  < ({\phi_\Mini}/{M})^\frac{1}{\g}$, and the contraction rate scales with $\t^\g$. 
 As in the proof of \Cref{lem: conv L-scheme} 
 we conclude that $u^i_n \to u_n$ in $L^2(\Om)$ and $w^i_n \to w_n$ in $H^1(\Om)$ by Banach's fixed-point theorem. 

The degenerate case is dealt with in the same manner as in the proof of \Cref{lem: conv L-scheme}. Analogous to \eqref{eq: L-scheme inequality}, we find using $L^i_n \geq 2M \t^\g$ and $\e:= \inf\left(L^i_n-\brac{\frac{\d \Phib}{\d u}}^{i-1}\right)\overset{\eqref{ineq: 2}}\geq (M-M_0)\t>0$ that
\begin{equation}
    M\t^\g  \int_\Om h_{n-1}|e^i_u|^2 + \frac{\varepsilon}{2}\int_{\Om}h_{n-1}|e^i_u - e^{i-1}_u|^2 + \t \norm{\nabla e^i_w}^2_{L^2(\Om)} \leq M\t^\g  \int_{\Om}h_{n-1}|e^{i-1}_u|^2 \label{eq: M-scheme inequality deg}.
\end{equation}
Summing both sides of \eqref{eq: M-scheme inequality deg} yields
\begin{equation}
\begin{split}
       0 &\leq \frac{\varepsilon}{2}\sum_{i=1}^N \brac{\int_{\Om}h_{n-1}|e^i_u - e^{i-1}_u|^2} + \t \sum_{i=1}^N \norm{\nabla e^i_w}_{L^2(\Om)}^2 \\
       &\leq M\t^\g  \brac{\int_{\Om}h_{n-1}|e^0_u|^2 - \int_{\Om}h_{n-1}|e^N_u|^2} < \infty. 
\end{split}
\end{equation}
This implies that $w^i_n \to w_n$ in $\H$, $\Phib(u^i_n) \to w_n = \Phib(u_n)$ in $L^2(\Om)$ and $u^i_n \to u_n$ in $L^1(\Om)$.  The result for the positive part of $u^i_n$ follows from \Cref{rem: un positive} which completes the proof.
\end{proof}

\begin{remark}
Following the proof for the L-scheme in Lemma \ref{lem: conv L-scheme} we would obtain the contraction rate
\begin{equation}
    \norm{(e^i_u, e^i_w)}_{M} \leq \sqrt{\frac{2M\t^\g }{\phi_\Mini + M\t^\g }} \norm{(e^{i-1}_u, e^{i-1}_w)}_{M},
\end{equation}
which is a larger than the rate in Lemma \ref{lem: conv M-scheme} if $\t < (\phi_\Mini / M)^\frac{1}{\g}$. However, this is the range of $\t$ where the M-scheme provides a contraction, and therefore \Cref{lem: conv M-scheme}, specifically \eqref{eq:M-contact-best}, provides a sharper result. Furthermore, the contraction rate stated in \Cref{lem: conv M-scheme} is half the contraction rate obtained for the $L^\infty$-norm in \Cref{lem: Linf conv Mscheme}. 
\end{remark}

\begin{proof}[Proof of \Cref{thm: conv M-scheme}]
    The existence and uniqueness of solutions $\{(u^i_n, w^i_n)\}_{i\in\N}\subset L^2(\Om)\times \H$ for \eqref{eq: linear weak} follows from \Cref{lem: exist L} and the convergence results from \Cref{lem: conv M-scheme}. 
\end{proof} 

\section{Numerical Results}\label{sec: alg}
We use the finite element method (FEM) to compute the solutions as it directly links to the weak form of the Equations \eqref{eq: time-discrete weak system} and \eqref{eq: linear weak}. The FEniCSx package in Python is used to solve the finite-dimensional problems \cite{AlnaesEtal2014, BasixJoss, ScroggsEtal2022}, and all the code is made available on GitHub\footnote{Link to the GitHub repository: \href{https://github.com/Rsmeets99/M_scheme_biofilm_PDE}{https://github.com/Rsmeets99/M\_scheme\_biofilm\_PDE}}. Let $\mathcal{T}$ denote the triangulation of $\Om$ and let $\mathcal{P}_{p}(\mathcal{T})$ be the space of element-wise polynomials of degree up to $p \in \mathds{N}$. We define the FEM solutions to be $\tilde{u}_{n,h}^i \in U_h \coloneqq \mathcal{P}_{0}(\mathcal{T})$ and $w_{n,h}, w^i_{n,h} \in V_h \coloneqq  \mathcal{P}_{1}(\mathcal{T}) \cap H^1_0(\Om)$. For $\mu=1$, we take the spatial approximation $v_{n,h}\in V_h$ of $v_n$ since in this case $v_n$ is differentiable, and $v_{n,h}\in U_h$ otherwise. This leads us to the following problem:

\begin{problem}[Finite element system] \label{prob: FEM system}
    Given $u_{n-1,h}, u^{i-1}_{n,h}, v_{n-1,h}$, find $\brac{\tilde{u}^i_{n,h}, w^i_{n,h}} \in Z$, such that
\begin{equation} \label{eq: FEM equation}
\begin{split}
        &\brac{h_{n-1} \tilde{u}^i_{n,h}}, \phi_h + \t \brac{\nabla w^i_{n,h}, \nabla \phi_h}=\brac{u_{n-1,h}, \phi_h} \\
         &\brac{L^i_n \tilde{u}^i_{n,h}-w^i_{n,h}, \xi_h} = \brac{L^i_n u^{i-1}_{n,h} - \Phi\brac{u^{i-1}_{n,h}}, \xi_h}
\end{split}
\end{equation}
    for all $\brac{\xi_h, \phi_h} \in Z$, where $Z$ denotes the (mixed) finite element space $Z = U_h \times V_h$. Afterwards, set $u^i_{n,h} = \posi{\tilde{u}^i_{n,h}}$.
\end{problem}

We iteratively solve $u^i_{n,h}$ and $w^i_{n,h}$ until the following stopping criteria is met:
\begin{equation} \label{eq: FEM error}
    \norm{\brac{e^i_{u,h}, e^i_{w,h}}} = \int_{\Om} L^i_n |e^i_{u,h}|^2 \dx + \t \norm{\nabla e^i_{w,h}}^2_{L^2(\Om)} < \text{tol},
\end{equation}
where  $e^i_{u,h} \coloneqq u^i_{n,h} - u^{i-1}_{n,h}$, $e^i_{w,h} \coloneqq w^i_{n,h} - w^{i-1}_{n,h}$ and $\text{tol} \in \R^+_\ast$ is some tolerance. Once the tolerance  is reached, we set $u_{n,h} = u^i_{n,h}$ and calculate $v_{n,h}\in V_h$ by solving 
\begin{equation} \label{eq: FEM vn PDE}
    (v_{n,h}, \eta_h) + \t \mu(D(u_{n,h}) \nabla v_{n,h}, \nabla \eta_h) = \t(g(u_{n,h}, v_{n-1,h}), \eta_h) + (v_{n-1, h},\eta_h),\qquad \forall \eta_h \in V_h,
\end{equation}
in the PDE-PDE case ($\mu=1$), or by solving $v_{n,h}\in U_h$
\begin{equation} \label{eq: FEM vn ODE}
    (v_{n,h}, \eta_h) = \t(g(u_{n,h}, v_{n-1,h}), \eta_h) + (v_{n-1, h},\eta_h),
    \qquad \forall \eta_h \in U_h,
\end{equation}
in the PDE-ODE case ($\mu=0$). 

Depending on the specific function $g$ we could update $v_{n,h}$ in the PDE-ODE case explicitly through
\begin{equation}
    v_{n,h} = v_{n-1,h} + \t g(u_{n,h}, v_{n-1,h}).
\end{equation}
However, in general we cannot guarantee that  $v_{n,h}\in U_h$, while Equation \eqref{eq: FEM vn ODE} provides a projection onto the correct space. The full algorithm is summarised in Algorithm \ref{alg: M-scheme}. 
Instead of solving the full system \eqref{eq: linear weak}, it is possible to eliminate $u$ from the equations and solve only for $w$. 
While faster to solve due to the reduced dimension of the resulting linear system, it does require correct projection operators and this leads to a modification of $M$ making it dependent on the mesh size $h$. More details are given in \cite[Section 5.1.2]{ThesisRobin}. 

\begin{algorithm}
\caption{M-scheme algorithm} \label{alg: M-scheme}
$t = t_{\text{start}}$ \;
\For{$t<T$}{
    error = 1 \;
    \While{$\text{error} > \text{tol}$}{
    Solve system \eqref{eq: FEM equation} from \Cref{prob: FEM system} for $\tilde{u}^i_{n,h}, w^i_{n,h}$ \;
    Set $u^i_{n,h} = \posi{\tilde{u}^i_{n,h}}$ \;
    Compute new error\;
    }
    set $u_{n,h} = u^i_{n,h}$ \;
    set $w_{n,h} = w^i_{n,h}$ \;
    \eIf{$\mu = 1$}{Compute $v_{n,h}$ in the PDE case using $u_{n,h}$ through solving equation \eqref{eq: FEM vn PDE}}
    {Compute $v_{n,h}$ in the ODE case using $u_{n,h}$ through solving equation \eqref{eq: FEM vn ODE} }
    $t = t+\t$ \;
}
\end{algorithm}

We test our scheme for 3 different problems: a porous medium equation, the biofilm PDE-PDE model and the biofilm PDE-ODE model. The goal is to get numerical convergence results, as well as to compare the performance of the M-scheme to the Newton. 

\begin{remark}[Newton scheme]
    The `true' Newton scheme with $M = 0$ may not converge in the degenerate case without regularization. To overcome this problem we use the M-scheme with a very small $M$ (e.g. $M=10^{-7} \ll \text{tol}$) which is still large enough so that the scheme converges in most cases. This is a form of a regularized Newton scheme \eqref{eq:different_schemes}.
\end{remark}

\begin{remark}[L-scheme]
For the porous medium equation, the L-scheme is at least an order of magnitude slower than the M-scheme, while for the biofilm models the L required for convergence is so large, that it becomes multiple orders of magnitude slower. Therefore, we will not show the results for the L-scheme in our comparison.
\end{remark}

\begin{remark}[Adaptive M-scheme]
When using the M scheme, in practice, 
it is beneficial to choose $M$ adaptively in each step using a posteriori estimators. 
In the biofilm case, with an adaptive scheme, the required $M$ increases if $u_n$ approaches 1 which ensures convergence, while $M$ is small and therefore the scheme is fast when $u_n$ is bounded away from 1.
Similar work for the L-scheme has been done in \cite{Stokke2023}. This is however beyond the scope of our current work.
\end{remark}

\subsection{Porous medium equation}\label{sec: MPME}
As a first test case, we consider the 1D porous medium equation with a reaction term   
\begin{equation}\label{eq: mod PME}
    \partial_t u = \Delta \brac{u^m} + \beta u,
\end{equation}
where $m > 1$ and $\b \in \R$ in a bounded interval $\Om \subset \R$ with homogeneous Dirichlet boundary conditions. In our notation, this corresponds to $b = \infty$, $\Phi(u) = u^m$ and $f(v) = \beta$. Note here that $h_{n-1} > 0$ if $\b < 1/\t$. It serves as a good benchmark problem as \eqref{eq: mod PME} has the exact solution $u(x,t) = e^{\b t} z(x,s)$, where $s = \frac{1}{\b (m-1)}e^{\b (m-1) t}$ and $z$ 
is the Barenblatt solution \cite{vazquez2007porous} given by
\begin{equation}
    z(x,t) = t^{-\frac{d}{d(m-1)+2}}\posi{C - \frac{m-1}{2m(d(m-1)+2)}\left|xt^{-\frac{1}{d(m-1)+2}} \right|^2}^{\frac{1}{m-1}}.
\end{equation}
The exact solution $u$ is H\"older continuous in time with exponent $\g = 1/(m-1)$. \\

We first verify the consistency of the time discretisation stated in \Cref{theo:time_convergence} by computing the error in the left-hand side of \eqref{eq:Conv-H1} for different values of $\t$. As initial condition, we take the exact solution $u$ evaluated at $t = 0.5$.  The results are shown in \Cref{fig: MPME 1D conv time} exhibiting an order of convergence 
between 0.5 and 1. 
Note that the results are independent of the choice of $M$ and $\g$ as long as the linearisation scheme converges.

\begin{figure}[h]
    \centering
    \includegraphics[scale = 0.5]{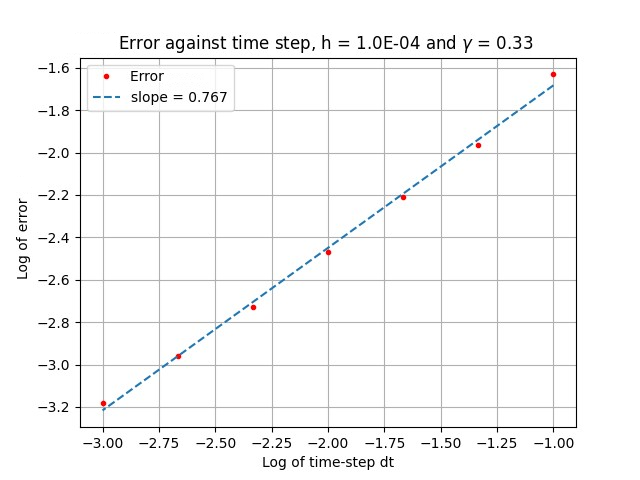}
    \caption{Error in \eqref{eq:Conv-H1} against time step size $\t$ for $h = 10^{-4}$, $m=4$, $\g = 1/3$, time $0.5 \leq t \leq 1$, $\text{tol} = 10^{-7}$, $d=1$, $\beta = 1$.}
    \label{fig: MPME 1D conv time}
\end{figure}

A convergence study of the iterative schemes for different time-steps $\t \in \{10^{-1}, 10^{-1.5},$ $10^{-2}, 10^{-2.5}\}$ was also performed with the mesh size $h$ ranging between $0.1$ and $0.005$. Then the average number of iterations needed to get to a final time $T = 1.1$ was determined for different values of $M \in \{10^{-1}, 10^{-2}, 10^{-3}, 10^{-7}\}$. The results are displayed in \Cref{fig: MPME 1D hdep}. 
We first note that for each $\t$ there is an optimal value $M$ and the M-scheme out-performs the Newton-scheme in this case, which is most apparent for smaller mesh sizes and larger time-steps. 
It is expected that the convergence of the Newton-scheme is conditioned by  restrictions on the time step size depending on the mesh size \cite{NS3}. For instance, for $\t = 10^{-1}$ the M-scheme performs significantly better than the Newton-scheme for small mesh sizes, while the schemes are equivalent for $\t = 10^{-2.5}$. Secondly, we note that the optimal value $M$ stays optimal for all mesh sizes. 
Hence, we can find the optimal $M$
for a coarse mesh, 
and use it then for computations on finer meshes \cite{erlend}. Lastly, the number of iterations required decreases with decreasing $\t$. The reason is two-fold: first, we expect the convergence rate to increase as $\t$ gets smaller by \Cref{thm: conv M-scheme}, and secondly, the difference between the solutions of two consecutive time steps $u_{n}$ and $u_{n-1}$ decreases when $\t$ does, and therefore the iterations start with a better initial guess $u^0_n = u_{n-1}$. 

\begin{figure}[h]
    \centering
    \begin{subfigure}{0.495\textwidth}
        \centering
        \includegraphics[width=\textwidth]{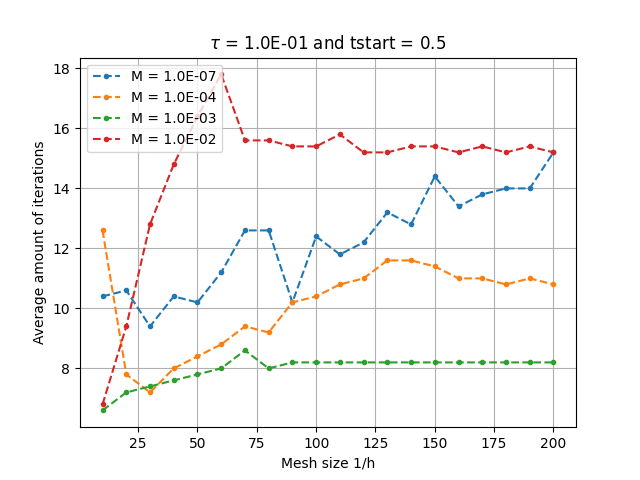}
        \caption[]
        {{\small $\t = 10^{-1}$}}    
        \label{fig: MPME 1D 0.1}
    \end{subfigure}
    \begin{subfigure}{0.495\textwidth}  
        \centering 
        \includegraphics[width=\textwidth]{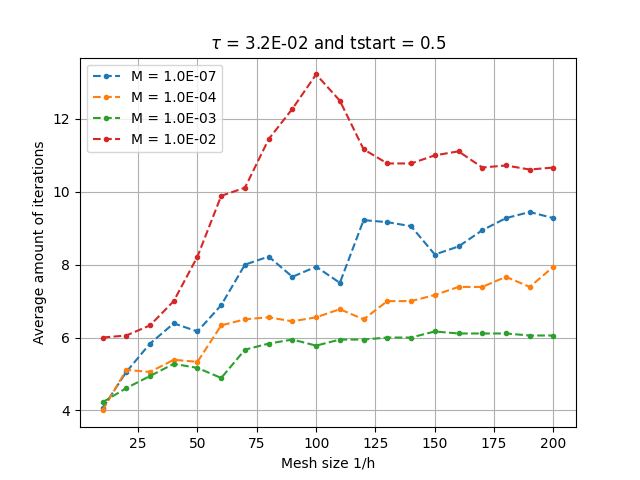}
        \caption[]
        {{\small $\t = 10^{-1.5}$}}    
        \label{fig: MPME 1D 0.03}
    \end{subfigure}
    \begin{subfigure}{0.495\textwidth}   
        \centering 
        \includegraphics[width=\textwidth]{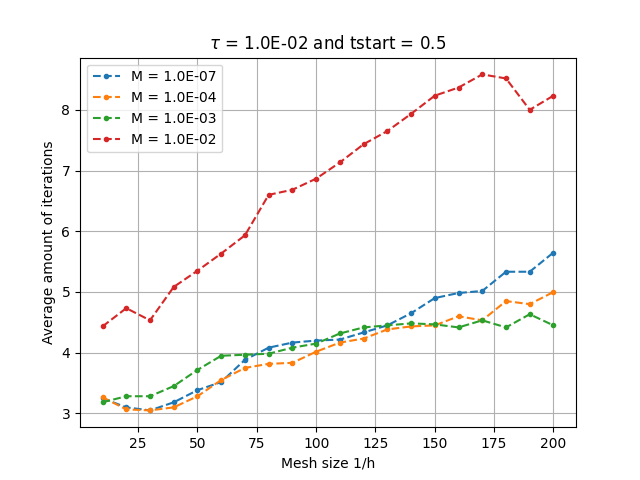}
        \caption[]
        {{\small $\t = 10^{-2}$}}    
        \label{fig: MPME 1D 0.01}
    \end{subfigure}
    \begin{subfigure}{0.495\textwidth}   
        \centering 
        \includegraphics[width=\textwidth]{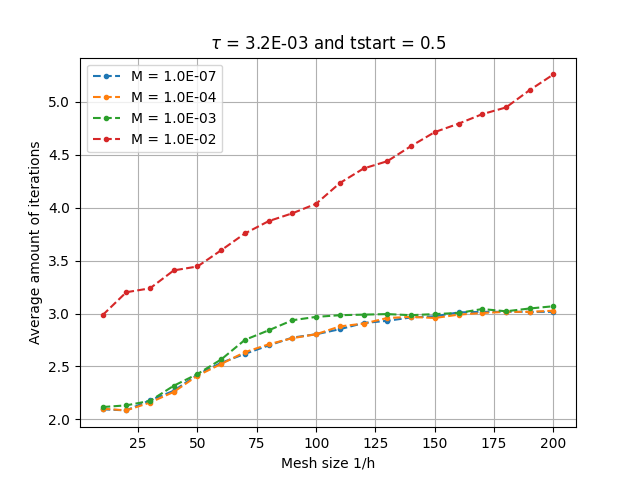}
        \caption[]
        {{\small $\t = 10^{-2.5}$}}    
        \label{fig: MPME 1D 0.003}
    \end{subfigure}
    \caption[]
    {\small Average iterations required for solving \eqref{eq: mod PME} in 1D for varying mesh size $h$ and time steps $\t$, with $m=4$, $\g = 1/3$, for time $0.5 \leq t \leq 1.1$, $\text{tol} = 10^{-5}$.} 
    \label{fig: MPME 1D hdep}
\end{figure}

Having found an optimal $M$ for $\g = 1/(m-1)$ (which is $M=10^{-3}$), we next test scaling of the contraction rates predicted by \Cref{thm: conv M-scheme}. Observe that not all assumptions are satisfied as the problem is degenerate. Nevertheless, we find a scaling of the contraction rate with some power of $\t$, as shown in \Cref{fig: MPME 1D conv gamma}. The contraction rate is calculated as the geometric mean of $\norm{(e^i_u,e^i_w)}/\norm{(e^{i-1}_u,e^{i-1}_w)}$ over the first 3 iterations. 
Note that the convergence rate $\a $ appears to scale linearly with $\t^{0.42}$ instead of $\t^{\g}=\t^{0.33}$. This `super-convergence' can be explained due to the fact that the challenging part of the numerical solution is the free boundary, while the solution is much more regular in the rest of the domain. For the test case in 1D, the free boundary consists of just two points. Hence, it does not play a deciding role in the convergence behaviour.

\begin{figure}[h]
    \centering
    \includegraphics[scale = 0.5]{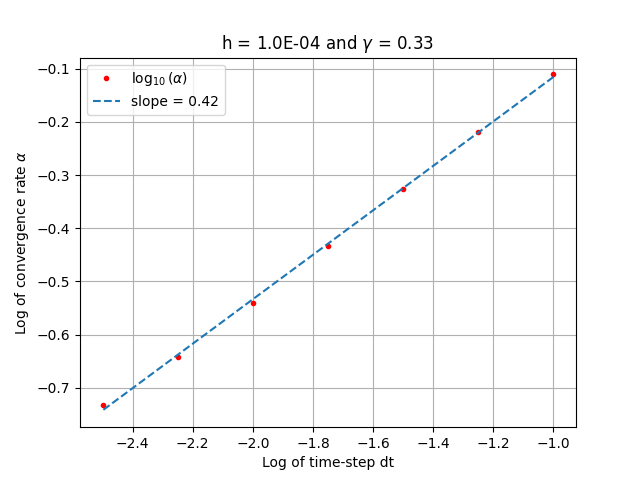}
    \caption{Convergence rate $\a$ against time step size $\t$ for $h = 10^{-4}$, $m=4$, $\g = 1/3$, $t = 0.5$, $d=1$.}
    \label{fig: MPME 1D conv gamma}
\end{figure}

\subsection{Biofilm equations}
In this section, we investigate the robustness of the M-schem for the more challenging biofilm models \eqref{eq: main} which are coupled systems involving a singular-degenerate diffusion equation. 
We consider the case $\mu = 1$ corresponding to a PDE-PDE coupling \cite{EBERLPDEPDE}, and $\mu = 0$ corresponding to a PDE-ODE coupling \cite{EBERLPDEODE}. The corresponding functions in \eqref{eq: main} are as follows:
\begin{subequations} \label{eq: PDE-PDE ex num}
    \begin{align}
        &\Phi^\prime(u) = d_1\frac{u^\a}{\brac{1-u}^\b}, &f(v) = k_3\frac{v}{v + k_2} - k_4, \\
        &D(u) = d_2, &g(u,v) = - k_1\frac{uv}{v+k_2},
    \end{align}
\end{subequations}
for some given constants $k_1,k_2,k_3,k_4,d_1,d_2 > 0$ and $\a,\b> 1$. For our comparison of the M-scheme and (regularized) Newton-scheme, we will use the same parameters as in \cite{EBERLPDEODE}, which are $k_1 = 0.4$, $k_2 = 0.01$, $k_3 = 1$, $k_4 = 0.42$, $d_1 = 10^{-6}$, $\a = 4$, $\b = 4$. 
For fixed $\a = \b = 4$, we can calculate $\Phi(u)$ explicitly, 
\begin{equation*} 
\begin{split}
    \Phi(u) = 10^{-6}\int_0^u \frac{s^4}{\brac{1-s}^4}ds = 10^{-6}\brac{ \frac{18u^2 - 30u + 13}{3\brac{1-u}^3} + u + 4\ln\brac{1-u} - \frac{13}{3}}.
\end{split}
\end{equation*} 

Note that we cannot simply use $\Phi$ but have to use its regularized form $\Phib$ as given in \Cref{def: reg Phi}. To define $\Phib$ we use 
 \Cref{thm: main time 1} to calculate the upper bound $\Breve{u}$. As the initial condition we take
\begin{equation}
    u_0(x) = \frac{h}{r} \left( \sqrt{\max(0, r^2 - (x - x_1)^2)} + \sqrt{\max(0, r^2 - (x - x_2)^2)} \right),
\end{equation}
with a maximum height $h = 0.9$, radius $r = 0.2$, $x_1 = -0.3$, $x_2 = 0.3$. For the domain $\Omega = (-1,1)$, this yields $\Breve{u} = 0.992$. 
For  $\g$ we take $\g=1/\a = 1/4$ since the regularity of solutions is expected to be similar as for the porous medium equation with the diffusion coefficient $\Phi^\prime \approx u^\a$ when $u$ is small (close to the free boundary). We assume homogeneous Neumann boundary conditions for $u$ and, if $\mu=1$, mixed boundary conditions for $v$. Namely, at the boundary $x=-1$ we specify the Dirichlet condition $v = 1$ and at $x=1$ homogeneous Neumann boundary conditions. While homogeneous Neumann conditions for $u$ are not covered by our theory, we still expect the results to hold as in the simulations, the biofilm region marked by the support of $u$, never reaches the boundary. The well-posedness results for time-continuous models in  \cite{HISSINKMULLER, mitra2023wellposedness} also apply to inhomogeneous Dirichlet and mixed boundary conditions, and to homogeneous Neumann conditions under certain time restrictions. We impose these boundary conditions as they were chosen for the numerical results in \cite{EBERLPDEODE, EBERLPDEPDE}, from where we also took our parameter values.

The results of the  M-scheme and (regularized) Newton-scheme are given in \Cref{fig: biofilm 1D mu0 hdep} for the PDE-ODE case. For the chosen parameters, the behaviour of the iterative schemes for the PDE-PDE case is almost identical. As  in \Cref{fig: MPME 1D hdep} we see that smaller time-steps $\t$ require fewer iterations for the reasons mentioned before. We only note that the amount of required iterations for the biofilm system is considerably higher. When solutions of the biofilm model approach 1, the diffusion coefficient $\Phib^\prime$ becomes very large, and therefore $L^i_n$ as well. This slows down the convergence of the iterative scheme.

A noticeable difference between the two figures is that if the mesh size is too small for the corresponding time-step, the Newton scheme becomes unstable in the biofilm case and often does not converge. As the time-step size decreases, the Newton scheme starts converging for smaller mesh sizes. However, we see a gap in the performance between the M-scheme and the Newton-scheme for these larger mesh sizes. The reason is that it is impossible to choose an optimal $M$ for the entire time range. A smaller $M$ would have a similar performance as the Newton-scheme for large mesh sizes, but as $u_n$ approaches values closer to 1, the diffusion coefficient blows up and convergence is no longer guaranteed for smaller mesh sizes. The choice of $M$ is therefore dictated by how close $u_n$ gets to 1. A larger $M$ improves stability at the cost of convergence speed. 

\begin{figure}[h]
    \centering
    \begin{subfigure}{0.495\textwidth}
        \centering
        \includegraphics[width=\textwidth]{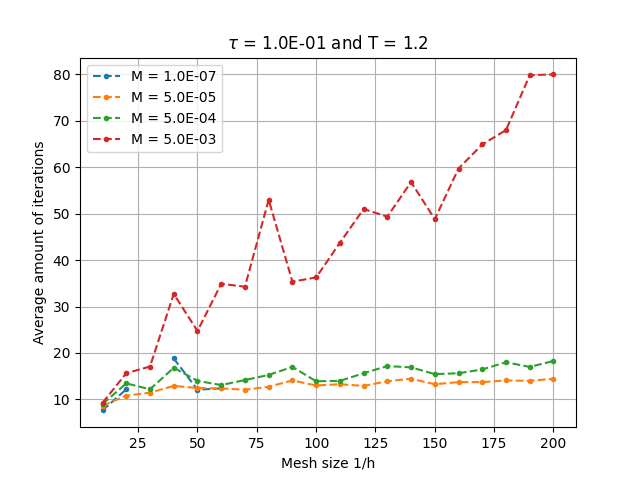}
        \caption[]
        {{\small $\t = 10^{-1}$}}    
        \label{fig: biofilm 1D mu0 0.1}
    \end{subfigure}
    \begin{subfigure}{0.495\textwidth}  
        \centering 
        \includegraphics[width=\textwidth]{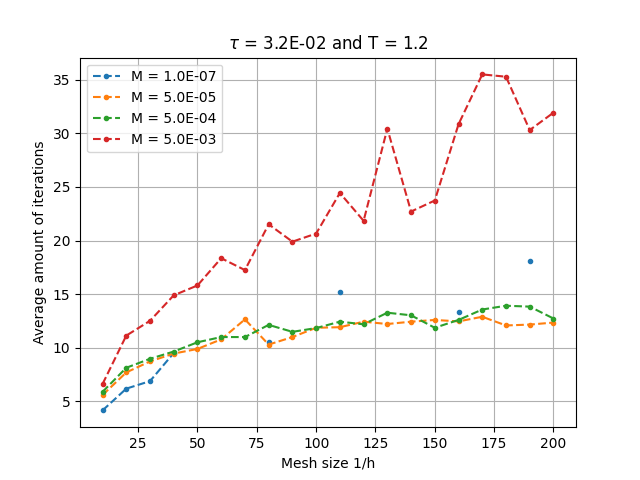}
        \caption[]
        {{\small $\t = 10^{-1.5}$}}    
        \label{fig: biofilm 1D mu0 0.03}
    \end{subfigure}
    \begin{subfigure}{0.495\textwidth}   
        \centering 
        \includegraphics[width=\textwidth]{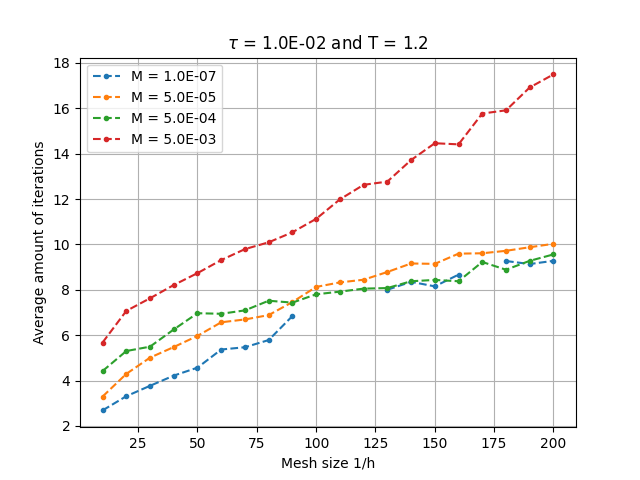}
        \caption[]
        {{\small $\t = 10^{-2}$}}    
        \label{fig: biofilm 1D mu0 0.01}
    \end{subfigure}
    \begin{subfigure}{0.495\textwidth}   
        \centering 
        \includegraphics[width=\textwidth]{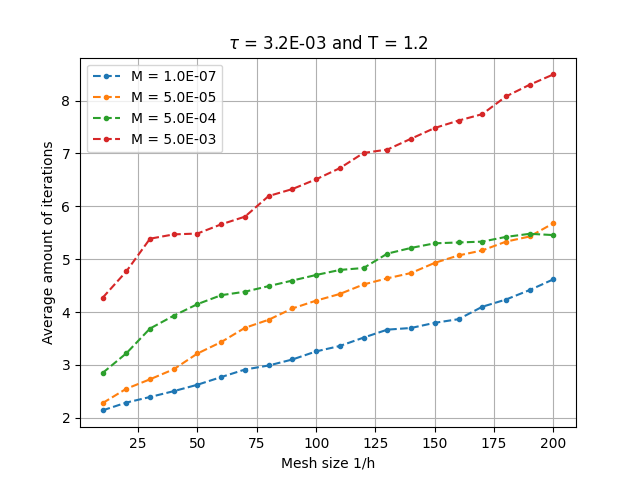}
        \caption[]
        {{\small $\t = 10^{-2.5}$}}    
        \label{fig: biofilm 1D mu0 0.003}
    \end{subfigure}
    \caption[]
    {\small Average iterations required for solving \eqref{eq: PDE-PDE ex num} in 1D for varying mesh size $h$ and time steps $\t$, with $m=4$, $\g = 1/4$, for time $0 \leq t \leq 1.2$, $\mu = 0$ and $\text{tol} = 10^{-5}$.} 
    \label{fig: biofilm 1D mu0 hdep}
\end{figure}

Similarly to \Cref{fig: MPME 1D conv gamma}, we calculate the contraction rate as the geometric mean over the first three iterations for different values of $\t$ in \Cref{fig: biofilm 1D conv gamma}. We find the contraction rate to be approximately $\t^{0.25}$, which aligns with our predicted value $\g$.

\begin{figure}[h!]
    \centering
    \includegraphics[scale = 0.6]{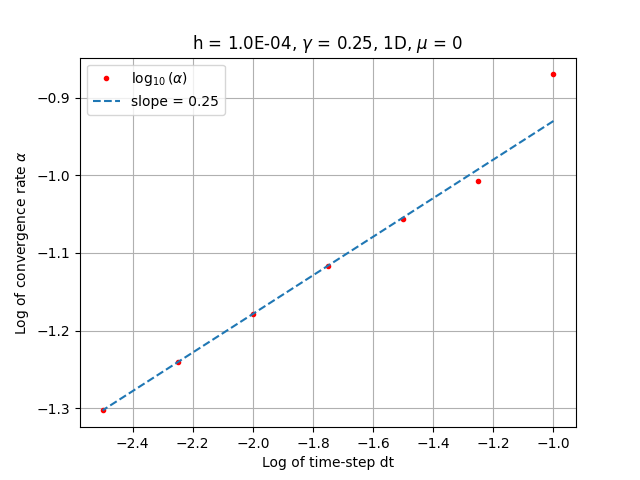}
    \caption{Convergence rate $\a$ against time step size $\t$ for $h = 10^{-4}$, $m=4$, $\g = 1/4$.}
    \label{fig: biofilm 1D conv gamma}
\end{figure}

A test example for both the PDE-PDE and PDE-ODE cases in two dimensions is given in \Cref{fig: PDE-PDE compl sim} and \Cref{fig: PDE-ODE compl sim} respectively. For the initial conditions, we have chosen two hemispheres as in the 1D case and used similar parameter values as disclosed in the caption of both figures. Computationally, these two problems are challenging. We find values of $u_n$ very close to $\Breve{u}$ ($\Breve{u} = 0.989$ and $\Breve{u} = 0.988$ for the PDE-PDE and PDE-ODE cases respectively), which leads to a blow-up of the diffusion coefficient. On top of that, the two blobs possess sharp interfaces that merge at some point creating additional singularities, see \Cref{fig: PDE-ODE compl sim}. For these reasons, the mesh size is kept relatively small to accurately resolve this merging. Despite these challenges, the numerical methods perform robustly, and we recover the expected qualitative behaviour of the solutions of both models.
In the PDE-PDE simulation, we see that since the nutrients $v_n$ diffuse and are constantly added through the Dirichlet boundary conditions on the top boundary, the biofilm expands towards the top, while slowly dying off at the bottom. For the PDE-ODE simulation, the biofilm expands in the radial direction as it consumes the nutrients while dying off in places where nutrients have been depleted. This leads to crater-like structures and inverse colony formation as seen in experiments, e.g. see \cite{EBERLPDEODE}.
Each 2D simulation required a long run-time, and due to the limitation of computational resources, a thorough comparison of iterative schemes could not be conducted in 2D.

\begin{figure}[p!]
    \centering
    \includegraphics[width=\textwidth]{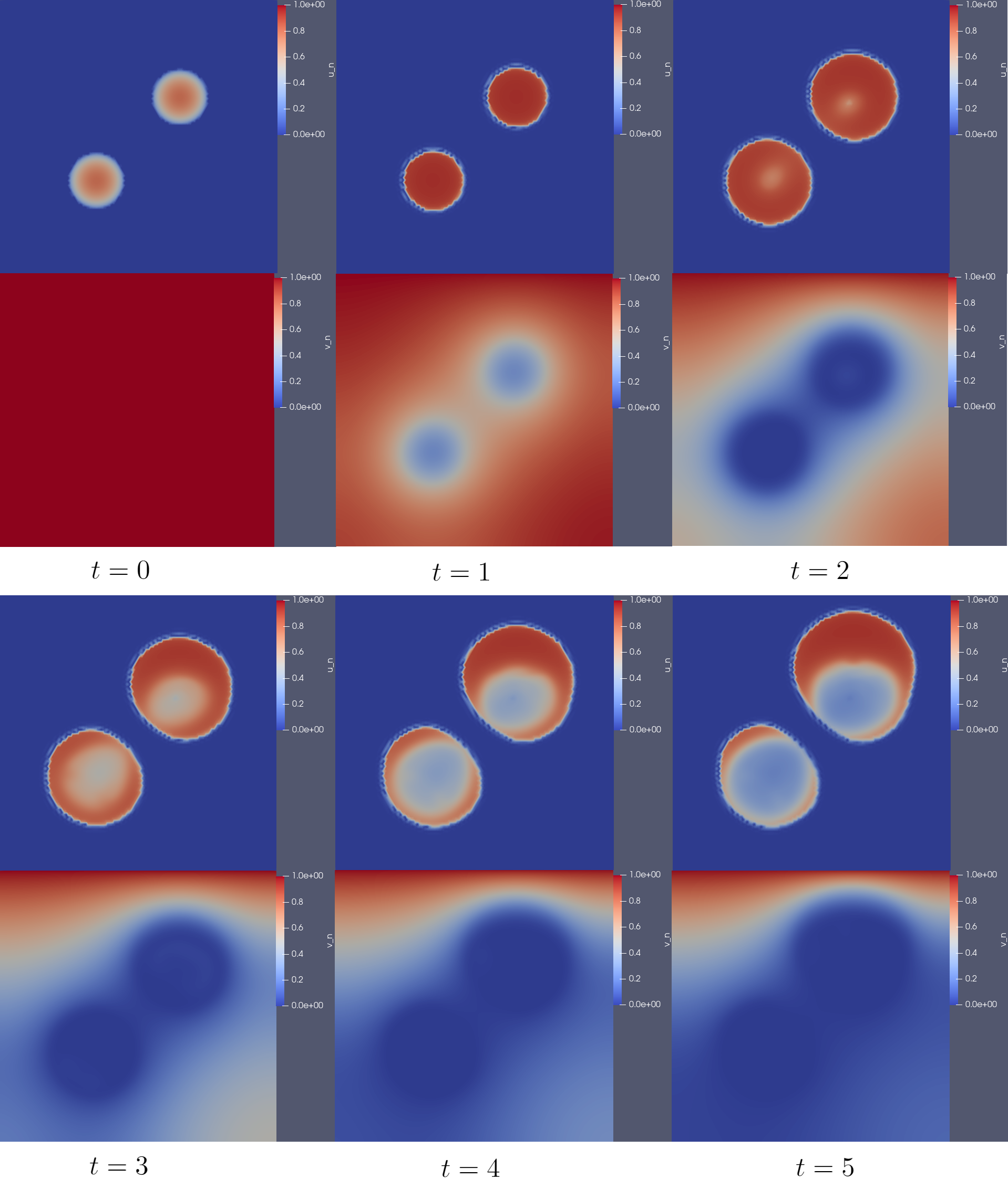}
    \caption{Simulation of the PDE-PDE model using the M-scheme ($M = 10^{-2})$, with $k_1 = 5, k_2 = 0.01, k_3 = 1, k_4 = 0.42, d_1 = 5\cdot 10^{-6}, d_2 = 0.2, \t = 0.01, h = 0.02$ and $\g = 0.5$. The first and third row picture $u_n$ while the second and fourth row $v_n$. For $v_n$ we have homogeneous Neumann boundary conditions at the sides and bottom, and the Dirichlet boundary condition $v=1$ at the top.}
    \label{fig: PDE-PDE compl sim}
\end{figure}

\begin{figure}[p!]
    \centering
    \includegraphics[width=\textwidth]{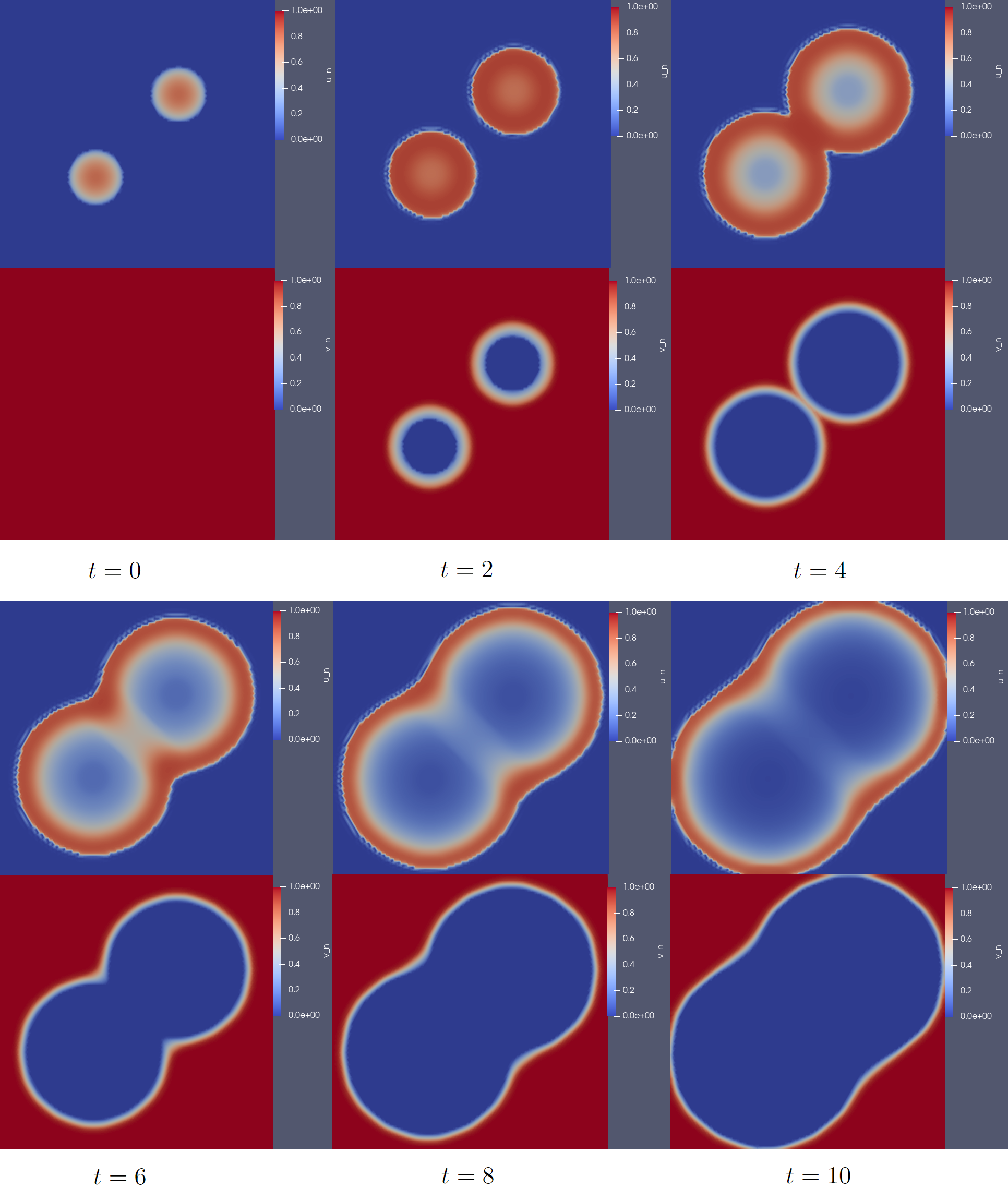}
    \caption{Simulation of the PDE-ODE model using the M-scheme ($M = 10^{-2})$, with $k_1 = 0.8, k_2 = 0.01, k_3 = 1, k_4 = 0.42, d_1 = 8\cdot 10^{-6}, \t = 0.01, h = 0.02$ and $\g = 0.5$. The first and third row picture $u_n$ while the second and fourth row $v_n$.}
    \label{fig: PDE-ODE compl sim}
\end{figure}

\section{Conclusion} \label{sec: concl and disc}

We introduced a semi-implicit time-discretisation scheme for solving a class of degenerate quasi-linear parabolic problems of porous medium type with diffusion coefficients that can also be singular. Such systems model biofilm growth and other nonlinear diffusion processes with sharp interfaces that propagate at a finite speed. The well-posedness of the time-discrete solutions was shown, as well as explicit upper bounds were proved for both the biofilm model ($b = 1$) and  porous medium equations ($b = \infty$). 
We then introduced the L/M-scheme as an iterative linearisation method for solving the quasi-linear elliptic PDEs that resulted from the time discretisation. It was shown that these schemes converge irrespective of the spatial discretisation. In the non-degenerate case, these schemes will even show a contraction with a contraction rate that scales with some power of $\t $ for the M-scheme provided $\t$ is small. Finally, the schemes were implemented numerically using a finite element method and it was shown that for larger time steps $\t$ and finer mesh sizes $h$, the M-scheme outperforms the Newton scheme.

The schemes can be generalised to systems that allow for additional substrates and admit terms for an advective flow field in these equations, see \cite{mitra2023wellposedness}. In a future work, we are considering including a nonlinearity in the time derivative as well which makes the problem doubly degenerate. Such problems are commonly found in multiphase flow through porous medium. Furthermore, one can consider multi-species biofilm models with or without cross-diffusion that comprise multiple degenerate equations that are strongly coupled through the diffusion operator, e.g. see \cite{ghasemi_sonner_eberl_2018, SoEfEb}.  

\section*{Acknowledgements}
K. Mitra acknowledges the support of Research Foundation - Flanders (FWO), the Junior Postdoctoral Fellowship grant  1209322N. K. Mitra and S. Sonner  thank the 
Nederlandse Organisatie voor  Wetenschappelijk
Onderzoek (NWO) for financial support (grant OCENW.KLEIN.358). The work of I.S. Pop was supported by FWO through the Odysseus programme (Project G0G1316N) and the German Research Foundation (DFG) through the SFB 1313, Project Number 327154368. 

\bibliographystyle{siam}

\end{document}